\documentclass[12pt]{article}
\usepackage{amsthm, amsmath, amssymb, amsfonts, graphicx, epsfig, bbm}
\usepackage{mathrsfs}
\usepackage{accents, comment}
\usepackage[numbers]{natbib}

\usepackage{hyperref}

\usepackage{amsfonts}
\usepackage[english]{babel}

\usepackage[T1]{fontenc}

\newtheorem{theorem}{Theorem}[section]
\newtheorem{lemma}[theorem]{Lemma}
\newtheorem{proposition}[theorem]{Proposition}
\newtheorem{corollary}[theorem]{Corollary}
\newtheorem{definition}[theorem]{Definition}
\newtheorem{remark}[theorem]{Remark}
\newtheorem{example}[theorem]{Example}

\newcommand{\Ker}{\operatorname{Ker}}

\usepackage{color}
\usepackage[normalem]{ulem}

\begin{document}
\title{On Function of Evolution of Distribution for Time Homogeneous Markov Processes}

\author{Tomasz R. Bielecki\footnote{ \small
  Department of Applied Mathematics,
	Illinois Institute of Technology,
	Chicago, IL 60616, USA
		e-mail: \url{tbielecki@iit.edu}},
		 Jacek Jakubowski\footnote{
		 	Institute of Mathematics, University
		 	of Warsaw,
		 	Banacha 2, 02-097 Warszawa, Poland,
		 	\url{J.Jakubowski@mimuw.edu.pl}}  and  Maciej Wi\'sniewolski\footnote{ corresponding author:
	 Institute of Mathematics, University
	of Warsaw,
	Banacha 2, 02-097 Warszawa, Poland,
	 	\url{ M.Wisniewolski@mimuw.edu.pl} }
}

\maketitle
\thispagestyle{empty}

\begin{abstract}
A study of time homogeneous, real valued Markov processes with a special property and a non-atomic initial distribution is provided. The new notion  of a function of evolution of distribution which determines the dependency between one dimensional distributions of a process is introduced. This, along with the notion of bridge operators which determine the backward structure, as opposed to the forward structure determined by the usual semi-group operators, paves a way to the new  approach for dealing with finite-dimensional distributions of Markov processes. This, in particular, produces explicit formulas which effectively simplify the computations of finite-dimensional distributions,  giving an alternative to the standard approach based on computations using the chain rule of transition densities. Various examples illustrating the new approach are presented.
\end{abstract}

\noindent
\begin{quote}
\noindent  \textbf{Key words}:  Markov processes, function of evolution of distribution, Feller processes, bridge operator, conditional dependence, Kolmogorov distance

\ \\
\textbf{AMS Subject Classification}:  60J35, 60H30, 91G80.

\end{quote}


\section{Introduction}

 We study a class of Markov processes with a non-atomic initial distribution and provide their special property. We introduce a new notion  of a function of evolution of distribution (in short FOED), which determines the dependency between one dimensional distributions of the process. We develop an original analysis based on the so called bridge operators which determine the backward structure which is  opposite to the forward structure determined by the usual semi-group operators.

The usual distributional description of a Markov process relies on determining the form of its transition density which by the classical chain rule gives the formula for finite dimensional distributions of the process (see \cite{RY}). If the transition density can not be determined explicitly the usual tool to find its approximate form involves Fokker-Planck equations. Some relatively new studies in this area can be find in \cite{Sar}, \cite{Lac} and in case of diffusions the older results are in  \cite{Sche}.

Development of various techniques to determine finite dimensional distributions of Markov processes has been motivated by their role in various applications, in particular in finance.  Surprisingly, the list of Markovian models for which such distributions have not yet been described explicitly is still quite vast. These include the so called stochastic volatility models, which are systems of stochastic differential equations of various types (\cite{Fo}, \cite{Fo2} \cite{Mu}, \cite{Ber}). While studying a class of stochastic volatility models Jakubowski and Wi\'sniewolski (\cite{JW09}, \cite{JW20})  discovered phenomena which we present and study in this work. Specifically, for Markov process with some non-trivial initial distribution we present a new way to describe the relations between its finite dimensional distributions. This is done by introduction of a function, which we call the function of evolution of distribution, and a related property which we call the FOED property. We derive several important consequences stemming from the FOED property. It  turns out that for the class of Feller processes having the FOED property a  one-dimensional distribution   can be determined explicitly in terms of the initial distribution and the generator. There is no need to look for the transition density.

The  FOED property for a stochastic process is strictly connected to the Markov property, and in general to properties of Markov families. The notion of FOED leads to a new backward approach to analysis of a Markov process in terms of so called bridge operators. Instead of looking at the starting point and continue with classical chain rule, the bridge operators pave a way to new iteration-type formulas describing the finite dimensional distributions of Markov processes. It turns out that under some assumptions about the underlying Markov process the computations of finite dimensional distributions of the process can be reduced to dealing with independent copies of the initial distribution of this process. Last but not least, the bridge operators and the FOED property deliver a new description of conditional distributions of Markov processes. 

The paper is organized as follows.
In Section 2 we define the FOED property, we present properties of processes with the FOED property and we give illustrating examples.
In Section 3 we investigate Feller processes with the  FOED property. In particular,  we prove that a  one-dimensional distribution   can be determined explicitly in terms of the initial distribution and the infinitesimal generator of the process.
In Section 4 we define a bridge operator and we observe that the finite dimensional distribution of a Markov process $X$, enjoying the FOED property, may be determined  in terms of bridge operators and the FOED property.  This gives a new algorithmic approach to the computational problems related to finite dimensional distributions of a Markov process. The algorithm uses the initial distribution of $X$ and a form of bridge operator. Using our methodology, in Section 5 we compute the expectation
$\mathbb{E}\big(\Psi\big(X_{0},\ldots,X_{t_n}\big)\big|  X_{s + t_1 + \ldots t_n} = z \big) $
for a Markov process $X$ with the FOED property, a   bounded Borel function $\Psi$, the discrete time tenor $0 < t_1 < \ldots  < t_n$ and $s\geq 0$.
In Section 6 we  present an application of  the FOED property to finding  the Kolmogorov distance
 between    $B_t$ and $X_t$, where $B$ is a Brownian motion with the  initial distribution $\gamma$ and  $X$ is a Markov process having the  FOED property and   the  initial distribution $\gamma$.

\section{Function of Evolution of Distribution}

We consider a complete probability space
$(\Omega,\mathcal{F},\mathbb{P})$. On this probability space we consider a  c\'adl\'ag
process $X$ with values in a Polish space $(E,\mathcal{E})$
and with the initial non-atomic distribution $\gamma$. In examples we usually have $E\subset \mathbb{R}^d$. The assumption about non-atomic  $\gamma$ is not restrictive since for a time-homogeneous Markov process $X$ the process defined by $(Y_t) = (X_{\epsilon + t})$, $\epsilon >0$ suits our setup. By $ \mathbb{F}^X=(\mathcal{F}^X_t)_{t\geq 0}$ we denote the filtration generated by $X$ and augmented to satisfy usual conditions. When we say that $X$ is a Markov process, we understand that $X$ is a Markov process with respect to $\mathbb{F}^X$.  We assume that the shift operator $\theta$ is well defined, so $X_s\circ \theta_t = X_{t+s}$.
 We start with the notion of a function of evolution of distribution.

\begin{definition} We say that process $X$ admits the  FOED property  if  there exists a Borel measurable  function
	$F:R_+\times E \mapsto R$ such that for every $y \in E$ the function $u\mapsto F(u,y)$  is  locally integrable  and
		\begin{align} \label{dfFED}
			\mathbb{E}f(X_t) = \mathbb{E} \big(f(X_0) \, e^{\int_0^tF(u,X_0)du} \big),
		\end{align}
	for every bounded Borel function $f$
	and every $t\geq 0$. We call such $F$  the function of evolution of distribution for $X$.
\end{definition}

Please note that distribution of the random variable $X_t$ is determined by values of $\mathbb{E}f(X_t)$ for bounded Borel functions $f$, or by values of $\mathbb{E}f(X_t)$ for positive Borel functions $f$. Thus, we may replace in the above definition the class of bounded Borel functions $f$ with the class of positive Borel functions $f$.

Moreover, if such $F$ exists then it is uniquely determined on the support of distribution of $X_0$.
Also, it should be noted that in general $F$ depends on the initial distribution $\gamma$, but for simplicity of notation we do not explicitly  exhibit this dependence. Let observe however that the notion of the function of evolution of distribution for $X$ is trivial, ergo not useful, in case of some initial distributions of $X$. For example, if $X_0 = x$ and $X$ possesses the  FOED property, then
	$X_t = x$ a.s. for every $t$. Indeed, taking $f\equiv 1$ shows that $e^{\int_0^tF(u,x)} = 1$ for every $t$ and (\ref{dfFED}) implies $\mathbb{E}f(X_t) = f(x)$ for every bounded Borel function $f$. For a nontrivial example, where the FOED property is not useful, it is enough to consider a $X_t = X_0 + B_t$, where $B$ is a standard Brownian motion and $X_0$ is a random variable independent of $B$, uniformly distributed on $[0,1]$.

As usual, given a process $X$,  then by  $\mathbb{P}_x$ for $x\in E$ we denote the probability measure defined on $\mathcal{F}$ by
\begin{align*}
	\mathbb{P}_x (A)= \mathbb{P}(A|X_0 = x), \quad A\in\mathcal{F},
\end{align*}
where the right hand side should be understood in the sense of the regular conditional probabilities.

The following example is a result of an earlier study, done by Jakubowski and Wi\'sniewolski in \cite{JW20}, of a class of SDEs that are widely used in finance. It turns out that processes given in terms of these SDEs posses the  FOED property. Actually, the paper \cite{JW20} initiated the study of the FOED property, motivation for which was the desire to simplify complexity of  calculations that traditionally plagued  applications  of the stochastic volatility models (as documented in, e.g., \cite{G}, \cite{Hag1}, \cite{Hul}, \cite{Jou}, \cite{RE}, \cite{Sin}).

\begin{example}
Consider a linear stochastic volatility model, where we assume that  the price $X_t$ at time $t$
of some underlying financial asset  has a stochastic volatility
$Y_t$, which has the dynamics given by a time dependent
 SDE,
so the dynamics of the vector $(X, Y)$ is given by
\begin{align}
    dX_t &= Y_tX_tdW_t, \label{defX} \\
    dY_t &= \mu(t,Y_t)\ dt + \sigma (t,Y_t)dZ_t, \label{defY}
\end{align}
with $X_0 = x$, $Y_0=y$ which are positive constants. Here the processes $W, Z$ are
correlated Brownian motions, $d{\left\langle W,Z\right\rangle}_t =
\rho dt$ with $\rho\in(-1,1)$,  and $\mu: \mathbb{R}_+ \times
\mathbb{R}_+ \rightarrow \mathbb{R}$, $\sigma: \mathbb{R}_+ \times
\mathbb{R}_+ \rightarrow \mathbb{R}$ are continuous functions such
that there exists a unique positive strong solution of \eqref{defY}. It turns out that for such a process $X$ and independent of it  gamma random variable  $\gamma$ with parameters $\alpha=1/2$ and $\beta=1$, {i.e. random variable with density $g(x) = \frac{1}{\sqrt{\pi x}}e^{-x}$}   we have
{\begin{align*}
	d\Big(\frac{X_t^2}{4\gamma}\Big) = 2\Big(\frac{X_t^2}{4\gamma}\Big)Y_tdW_t +  \Big(\frac{X_t^2}{4\gamma}\Big)Y_t^2dt
\end{align*}
and
\begin{equation*}
	\mathbb{P}\Big(\frac{X_t^2}{4\gamma}\in dz\Big) =	\mathbb{P}\Big(\frac{x^2}{4\gamma}\in dz\Big)e^{\int_0^tF(u,z)du}
\end{equation*}
}
(\cite[Thm. 3.5, Corr. 3.6]{JW20}). This means  that  process $\frac{X_t^2}{4\gamma}$ has the  FOED property.
The use of the FOED property also gives
\begin{align*}
	\mathbb{E}\Big(X_te^{-\lambda X_t^2}\Big) = x \exp\Big(-\lambda x^2 + \int_0^tF(u,1/4\lambda)du\Big), \quad \lambda > 0,
\end{align*}
where $F$ is the function of evolution of distribution of $\frac{X_t^2}{4\gamma}$
(\cite[Prop.3.7]{JW20}). This shows that determining $F$ gives the explicit formula for a (slightly transformed) Laplace transform of the random variable $X_t^2$. 
\end{example}

Since we are interested in time homogenous Markov processes we start with a proposition giving a broad class of such  processes admitting the FOED property for each initial distribution  with strictly positive density.
In particular, the proposition gives a way to construct the FOED function $F$. By $\mathcal{C}^1[0,\infty)$ we understand the class of functions defined on $[0,\infty)$ which belong to $\mathcal{C}^1(0,\infty)$ and have the right-continuous derivative at $0$.


\begin{proposition}
\label{Ex1}
	Let $X$ be a Markov process with the transition density $p$ with respect to the Lebesgue measure such that $t\mapsto p_t(x,y)\in\mathcal{C}^1[0,\infty)$ for $x,y \in E$.
Assume that $X$ has  the initial distribution with the strictly positive density $g$  (with respect to the Lebesgue measure).
Then $X$ has the FOED property and
\begin{equation}\label{eq:F} e^{\int_0^tF(u,x)du} = \frac{\int_Ep_t(z,x)g(z)dz}{g(x)},\ \textrm{for}\ t\geq 0.
\end{equation}
\end{proposition}

\begin{proof}
		Let $\phi(t,x)=\int_E p_t(z,x) g(z)dz,\ t>0$, $\phi(0,x)=g(x)$. Put
	\begin{align}\label{eq:tom7}
		F(t,x) = \frac{\partial \phi(t,x)}{\partial t}, \qquad t>0,
	\end{align}
and $F(0,x) = \frac{1}{g(x)}\frac{\partial \phi(t,x)}{\partial t}|_{t=0^+}$.	Then for any $t>0$ and $x\in E$ we have
	\begin{align}\label{eq:tom8}
	e^{\int_0^t F(u,x)du} = \frac{\phi(t,x)}{g(x)},
	\end{align}
which implies \eqref{dfFED},  so $X$ has the FOED property with function $F$ satisfying \eqref{eq:F}.
 \end{proof}

\begin{example} \label{Ex1a}
	 Let $X$ be as in Proposition \ref{Ex1} and  $X_0 = x \in E$. Let $s>0$ be such that the distribution of $X_s$ has the strictly positive density $g_s$.
For this $s>0$, we define a new process $Y_t = X_{t+s}$, $t\geq 0$. Then the initial distribution of $Y$ is $\gamma(dz) =  g_s(z)dz$.
It turns out that under the above assumptions the process $Y$ has the  FOED property. Indeed, for any bounded Borel $f$ on $E$ and $t> 0$
\begin{align*}
	\mathbb{E}f(Y_t) &= \mathbb{E}\mathbb{E}_{X_s}f(X_t) = \int_{E\times E}f(z)p_t(v,z)g_s(v)dvdz
	=  \int_{E}f(z)e^{\int_0^t F(u,z)du}g_s(z)dz,
\end{align*}
where, in line with the proof  of Proposition \ref{Ex1}, we take $F$ such that
\begin{align*}
e^{\int_0^t F(u,z)du} =  \frac{\int_E p_t(v,z) g_s(v)dv}{g_s(z)}.
\end{align*}
\end{example}

Other examples of processes with the FOED property will be given later.

 In the sequel we will use the following notation
\begin{align*}
\mu_t(dz) = \mathbb{P}(X_t\in dz), \quad t >0.
\end{align*}

The next proposition gives two immediate consequences of the FOED property.

\begin{proposition} \label{prop2.4}
Suppose that  $X$ has the FOED property. Then,

a)	 For $0\leq s<t$ and for every positive Borel function $f$
	\begin{equation*}
		\mathbb{E}f(X_t) = \mathbb{E} \big(f(X_0) \, e^{\int_s^t F(u,X_0)du} e^{\int_0^s F(u,X_0)du}\big) = \mathbb{E} \big(f(X_s) \, e^{\int_s^t F(u,X_s)du} \big).
	\end{equation*}
b)	The function  $F$ determines uniquely the density of $\frac{d\mu_t}{d\mu_s} $, namely
	\begin{align}\label{abscon}
	\mu_t(dz) = \mu_s(dz)e^{\int_s^tF(u,z)du},
	\end{align}
in particular
$\mu_t(dz) = \gamma(dz)e^{\int_0^tF(u,z)du}$,
so $E$ is the support of $\mu_t$ provided that $E$ is the support of $\mu_0$. This means the FOED property implies absolute continuity of one dimensional distributions of $X$.
\end{proposition}

Another immediate consequence of the  FOED property is

\begin{align}\label{cons}
\mathbb{E}\big(f(X_t)e^{-\int_0^tF(u,X_t)du}\big) = \mathbb{E}f(X_0).
\end{align}
Taking $f \equiv 1$ we obtain
	\begin{align*}
	\mathbb{E} e^{-\int_0^tF(u,X_t)du} = 1.
	\end{align*}

\begin{theorem} \label{m_FOD}
	 Assume that process  $X$ on $(\Omega,\mathcal{F},\mathbb{P})$ has the FOED property and the inital  distribution with strictly positive density.
	 The following statements are equivalent:\\
	(i) For each $\mathcal{F}^X_{\infty}$ measurable positive or bounded random variable $Z$ and for each $t>0$ and $x, y\in E$ it holds
	\begin{align*}
		\mathbb{E}(Z\circ \theta_t) = \mathbb{E}Ze^{\int_0^tF(u,X_0)du}, \qquad \mathbb{E}(Z\circ \theta_t|X_t = x, X_0 = y) = \mathbb{E}(Z\circ \theta_t|X_t = x).
	\end{align*}
	(ii) $X$ is a Markov process with respect to $\mathbb{F}^X$.
\end{theorem}
\begin{proof}
	$(i)\Rightarrow (ii)$
It is enough to prove  for arbitrary $n\in \mathbb{N}$, $0 \leq t_1 < t_2 < \ldots < t_n < t$ and a positive, Borel function $f$ on $E$ that
\begin{align}\label{ndim}
	\mathbb{E}[f(X_{t})|X_{t_1},\ldots,X_{t_n}] = \mathbb{E}_{X_{t_n}}f(X_{t-t_n}).
\end{align}
Indeed, from \eqref{ndim} using the monotone class theorem we see that
\begin{align*}
	\mathbb{E}_{X_{t_n}} [f(X_{t-t_n}) ]= \mathbb{E}[f(X_{t})|\mathcal{F}^X_{t_n}].
\end{align*}
Hence  for any  $A\in \mathcal{F}^X_{t_n}$
\begin{align*}
	\mathbb{E}_{X_0}[f(X_{t})1_A] &= \mathbb{E}[f(X_{t})1_A|X_0] = \mathbb{E}[\mathbb{E}[f(X_{t})1_A|X_0]|\mathcal{F}^X_{t_n}]= \mathbb{E}[1_A\mathbb{E}[f(X_{t})|\mathcal{F}^X_{t_n}]|X_0] \\
	&= \mathbb{E}[1_A\mathbb{E}_{X_{t_n}}[f(X_{t-t_n})]|X_0] = \mathbb{E}_{X_0}[1_A\mathbb{E}_{X_{t_n}}[f(X_{t-t_n})]],
\end{align*}
which  yields for every $x\in E$
\begin{align}\label{eq:family}
	\mathbb{E}_{x}[f(X_{t})1_A] = \mathbb{E}_{x}[1_A\mathbb{E}_{X_{t_n}}[f(X_{t-t_n})]].
\end{align}
Hence we obtain for $s\leq t$ that $$\mathbb{P}_x(X_t\in A|\mathcal{F}_s)=\mathbb{P}_{X_s}(X_{t-s}\in A) \quad \mathbb{P}_x \ a.s.$$
Integrating both sides with respect to the initial distribution yields
$$\mathbb{P}(X_t\in A|\mathcal{F}_s)=\mathbb{P}_{X_s}(X_{t-s}\in A) \quad \mathbb{P} \ a.s.$$
showing that $X$ is a Markov process.

So it is enough to prove \eqref{ndim}. We start with showing that for any  $\mathcal{F}^X_{\infty}$ measurable, positive or bounded random variable $Z$ we have
\begin{align}\label{z1}
	\mathbb{E}_{x}Z = \mathbb{E}[Z\circ\theta_v|X_v = x], \qquad x \in E,\ v\geq 0.
\end{align}
	 Let $f$ and $g$ be positive, Borel functions on $E$ and $t,v \in [0,\infty)$. Define
	 $h$ by formula
	 $$h(x,v,t + v) = \mathbb{E}(f(X_{t+v})|X_v = x).$$
	  The function $h$ is well defined since $E$ is, by assumption, the support of distribution of $X_u$ for each $ u\geq 0$. From the definition of FOED we have
\begin{align}\label{eq:tom1}
	\mathbb{E}[f(X_{t+v})g(X_v)] = \mathbb{E}[h(X_v,v,t+v)g(X_v)]= \mathbb{E}\big(h(X_0,v,t+v)g(X_0)e^{\int_0^vF(u,X_0)du}\big).
\end{align}
 Note that for $Z=f(X_{t})g(X_0)$ we have that $Z\circ \theta_v= f(X_{t+v})g(X_v)$. Thus, if $(i)$ holds, then
\begin{align}\label{eq:tom2}
	\mathbb{E}[f(X_{t+v})g(X_v)] &= \mathbb{E}\big(f(X_{t})g(X_0)e^{\int_0^vF(u,X_0)du}\big).
\end{align}
Comparing \eqref{eq:tom1} and \eqref{eq:tom2} we see that (almost surely) $f(X_{t})=h(X_0,v,t+v)$, so that
\begin{align*}
	\mathbb{E}[f(X_{t})|X_0 = x] = h(x,v,t+v),
\end{align*}
and so, actually, $ h(x,0,t)=h(x,v,t + v)$. Therefore
\begin{align}\label{1dim}
	\mathbb{E}_{x}f(X_{t}) = \mathbb{E}[f(X_{t+v})|X_v = x].
\end{align}

Proceeding in the similar way we show \eqref{z1}.
It is enough to prove \eqref{z1} for $Z = \Pi_{i=1}^n f_i(X_{t_i})$, where  $n \in \mathbb{N}$, $0 \leq v < t_1 < t_2 < \ldots < t_n $ and $f_i$, $i=1,\ldots,n$, are positive Borel functions,  and then use the monotone class theorem.
Let $\widehat{h}$ be defined by
\begin{align*}
	\widehat{h}(x,v,t_1,\ldots,t_n) = \mathbb{E}\Big[\Pi_{i=1}^nf_i(X_{t_i})|X_v = x\Big].
\end{align*}
 For a positive, Borel function $g$ we have
\begin{align}\label{eq:tom3}
	\mathbb{E}\Big[g(X_v)\Pi_{i=1}^nf_i(X_{t_i})\Big] &= \mathbb{E}\Big[g(X_v)\widehat{h}(X_v,v,t_1,\ldots,t_n)\Big]\\
	&= \mathbb{E}\Big[g(X_0)\widehat{h}(X_0,v,t_1,\ldots,t_n)e^{\int_0^vF(u,X_0)du}\Big], \nonumber
\end{align}	
where in the last equality we used the definition of FOED.
On the other hand by $(i)$
\begin{align}\label{eq:tom4}	
\mathbb{E}\Big[g(X_v)\Pi_{i=1}^nf_i(X_{t_i})\Big] &= \mathbb{E}\Big[\Big(g(X_0)\Pi_{i=1}^nf_i(X_{t_i-v})\Big)\circ\theta_v\Big]\\
 &= \mathbb{E}
 \Big[g(X_0)\mathbb{E}\Big[\Pi_{i=1}^nf_i(X_{t_i-v})|X_0\Big]e^{\int_0^vF(u,X_0)du}\Big]. \nonumber
\end{align}
So comparing \eqref{eq:tom3} and \eqref{eq:tom4}  we have for any $x\in E$
\begin{align*}
	\widehat{h}(x,v,t_1,\ldots,t_n) = \mathbb{E}\Big[\Pi_{i=1}^nf_i(X_{t_i-v})|X_0 = x\Big] = \mathbb{E}_{x}\Big[\Pi_{i=1}^nf_i(X_{t_i-v})\Big].
\end{align*}
As a result, recalling the definition of $\widehat{h}$ and using the last equality we obtain
\begin{align*}
	\mathbb{E}\Big[\Pi_{i=1}^nf_i(X_{t_i})|X_v = x \Big] = \mathbb{E}_{x}\Big[\Pi_{i=1}^nf_i(X_{t_i-v})\Big],
\end{align*}
and (\ref{z1}) follows  for $Z = \Pi_{i=1}^n f_i(X_{t_i})$.

We now show (\ref{ndim}). 
 Let
\begin{align*}
	h_1(t,t_1,\ldots,t_n, x_1,\ldots,x_n) = \mathbb{E}[f(X_{t})|X_{t_1}=x_1,\ldots,X_{t_n}=x_n] .
\end{align*}
We see from $(i)$ that for arbitrary positive Borel functions $f_1,\ldots,f_n$
\begin{align}\label{eq:tom5}
	\mathbb{E}&[f(X_{t})f_1(X_{t_1})\ldots f_n(X_{t_n})]
	= \mathbb{E}[h_1(t, t_1,\ldots,t_n, X_{t_1},\ldots,X_{t_n})f_1(X_{t_1})\ldots f_n(X_{t_n})]\\ \nonumber
	&= \mathbb{E}f_1(X_0)e^{\int_0^{t_1}F(X_0, u)du}h_1(t, t_1,\ldots,t_n, X_0,X_{t_2-t_1},\ldots,X_{t_n-t_1})\Pi_{i=2}^nf_i(X_{t_i-t_1})\\
	&= \mathbb{E}\Big[f_1(X_0)e^{\int_0^{t_1}F(X_0, u)du}\mathbb{E}_{X_0}\Big(h_1(t, t_1,\ldots,t_n, X_0,X_{t_2-t_1},\ldots,X_{t_n-t_1})\Pi_{i=2}^nf_i(X_{t_i-t_1})\Big)\Big], \nonumber
\end{align}
where in the second equality we used identity (\ref{z1}). On the other hand defining $\widetilde{h}$ by
\begin{align*}
	\widetilde{h}(t, t_1,t_2,\ldots,t_n,x) = \mathbb{E}\Big[f(X_t)\Pi_{i=2}^nf_i(X_{t_i})|X_{t_1}=x\Big]
\end{align*}
we have
\begin{align}\label{eq:tom6}
	\mathbb{E}[f(X_{t})f_1(X_{t_1})\ldots f_n(X_{t_n})] &= \mathbb{E}\Big[ f_1(X_{t_1})\widetilde{h}(t, t_1, t_2,\ldots,t_n,X_{t_1})\Big]\\
	&= \mathbb{E}\Big[ f_1(X_{0})\widetilde{h}(t, t_1, t_2,\ldots,t_n,X_{0})e^{\int_0^{t_1}F(X_0, u)du}\Big], \nonumber
\end{align}
where in the last equality we used the definition of FOED. As a result comparing \eqref{eq:tom5} and \eqref{eq:tom6} we obtain for any $x_1\in E$
\begin{align}\label{z2}
	\widetilde{h}(t, t_1, t_2,\ldots,t_n,x_1) = \mathbb{E}_{x_1}\Big(h_1(t,t_1,\ldots,t_n,x_1,X_{t_2-t_1},\ldots,X_{t_n-t_1})\Pi_{i=2}^nf_i(X_{t_i-t_1})\Big).
\end{align}
Observe now that  by (\ref{z1}), definition of $\widetilde{h}$,  and then (\ref{z2}) we have
\begin{align}
	& \mathbb{E}_{x_1}\Big[f(X_{t-t_1})\Pi_{i=2}^nf_i(X_{t_i - t_1})\Big]
	 = \mathbb{E}\Big[f(X_t)\Pi_{i=2}^nf_i(X_{t_i})|X_{t_1}=x_1\Big]= \widetilde{h}(t, t_1, t_2,\ldots,t_n,x_1) \notag \\
	&
\label{z3}
= \mathbb{E}_{x_1}\Big(h_1(t,t_1,\ldots,t_n,x_1,X_{t_2-t_1},\ldots,X_{t_n-t_1})\Pi_{i=2}^nf_i(X_{t_i-t_1})\Big).
\end{align}
We write the left-hand side of (\ref{z3}) as follows:
\begin{align*}
	\mathbb{E}_{x_1}\Big[f(X_{t-t_1})\Pi_{i=2}^nf_i(X_{t_i - t_1})\Big] = \mathbb{E}_{x_1}\Big(f_2(X_{t_2-t_1})\mathbb{E}_{x_1}\Big[f(X_{t-t_1})\Pi_{i=3}^nf_i(X_{t_i - t_1})|X_{t_2-t_1}\Big]\Big).
\end{align*}
On the other hand, the right-hand side of (\ref{z3}) is
\begin{align*}
	\mathbb{E}_{x_1}&\Big(h_1(t,t_1,\ldots,t_n,x_1,X_{t_2-t_1},\ldots,X_{t_n-t_1})\Pi_{i=2}^nf_i(X_{t_i-t_1})\Big)\\ &= \mathbb{E}_{x_1}\Big(f_2(X_{t_2-t_1})\mathbb{E}_{x_1}\Big[h_1(t,t_1,\ldots,t_n,x_1,X_{t_2-t_1},\ldots,X_{t_n-t_1})\Pi_{i=3}^nf_i(X_{t_i-t_1})|X_{t_2-t_1}\Big]\Big).
\end{align*}
So from (\ref{z3}) we have for any $x_2\in E$
\begin{align}\label{pom1}
\mathbb{E}_{x_1}&\Big[f(X_{t-t_1})\Pi_{i=3}^nf_i(X_{t_i - t_1})|X_{t_2-t_1} = x_2\Big]\\
 &= \mathbb{E}_{x_1}\Big[h_1(t,t_1,\ldots,t_n,x_1,x_2,\ldots,X_{t_n-t_1})\Pi_{i=3}^nf_i(X_{t_i-t_1})|X_{t_2-t_1} = x_2\Big].\notag
\end{align}
Observe that the last equality may be rewritten as
\begin{align*}
\mathbb{E}&\Big[\Big(f(X_{t-t_2})\Pi_{i=3}^nf_i(X_{t_i - t_2})\Big)\circ\theta_{t_2 - t_1}\big|X_{t_2-t_1} = x_2, X_0 = x_1\Big]\\
 &= \mathbb{E}\Big[\Big(h_1(t,t_1,\ldots,t_n,x_1,x_2,X_{t_3-t_2},\ldots,X_{t_n-t_2})\Pi_{i=3}^nf_i(X_{t_i-t_2})\Big)\circ\theta_{t_2 - t_1}\big|X_{t_2-t_1} = x_2, X_0 = x_1\Big],
\end{align*}
so by the second condition in (i) we get
\begin{align*}
\mathbb{E}&\Big[\Big(f(X_{t-t_2})\Pi_{i=3}^nf_i(X_{t_i - t_2})\circ\theta_{t_2 - t_1}\Big)\big|X_{t_2-t_1} = x_2\Big]\\
 &= \mathbb{E}\Big[\Big(h_1(t,t_1,\ldots,t_n,x_1,x_2,X_{t_3-t_2},\ldots,X_{t_n-t_2})\Pi_{i=3}^nf_i(X_{t_i-t_2})\Big)\circ\theta_{t_2 - t_1}\big|X_{t_2-t_1} = x_2\Big].
\end{align*}
Using  (\ref{z1}) we obtain
\begin{align*}
	\mathbb{E}_{x_2}\Big[&f(X_{t-t_2})\Pi_{i=3}^nf_i(X_{t_i - t_2})\Big]\\
	&= \mathbb{E}_{x_2}\Big[h_1(t,t_1,\ldots,t_n,x_1,x_2,X_{t_3-t_2},\ldots,X_{t_n-t_2})\Pi_{i=3}^nf_i(X_{t_i-t_2})\Big].
\end{align*}
Repeating the last procedure $(n-3)$ times, we obtain
\begin{align*}
	\mathbb{E}_{x_n}\Big[f(X_{t-t_n})\Big] = h_1(t,t_1,\ldots,t_n,x_1,x_2,\ldots,x_n),
\end{align*}
which after recalling the definition of $h$ yields
\begin{align*}
 \mathbb{E}_{X_{t_n}}\Big[f(X_{t-t_n})\Big]= \mathbb{E}[f(X_{t})|X_{t_1},\ldots,X_{t_n}].
\end{align*}
This proves \eqref{ndim}. The proof of implication $(i)\Rightarrow (ii)$ is finished.

\ \\
 $(ii)\Rightarrow (i)$ The second condition in (i) follows directly from the Markov property, so we only need to show the first one. Using the argument of monotone-class theorem it is enough to prove  that
\begin{align} \label{j9-1}
	 \mathbb{E} \big( \Pi_{i=1}^nf_i(X_{t_i+t}) \big)=\mathbb{E} \Big( \big ( \Pi_{i=1}^nf_i(X_{t_i}) \big)\circ \theta _t \Big ) = \mathbb{E} \big(\Pi_{i=1}^nf_i(X_{t_i})e^{\int_0^tF(X_0, u)du} \big)
\end{align}
 for arbitrary $0<t_1<\ldots<t_n$ and $f_1,\ldots,f_n$ positive Borel functions on $E$.

 The proof goes by induction. For $n=1$ the assertion follows directly by the definition of FOED. Indeed, using Markov and FOED properties we obtain that
\begin{align*}
	\mathbb{E}&f_1(X_{t_1+t}) = \mathbb{E} \mathbb{E}\big(f_1(X_{t_1+t})|\mathcal{F}^X_t\big) = 
	\mathbb{E}\mathbb{E}_{X_t}f_1(X_{t_1})\\ & = \mathbb{E}\big( e^{\int_0^tF(X_0, u)du}\mathbb{E}_{X_0}f_1(X_{t_1})\big)
	= \mathbb{E}\mathbb{E}\big( e^{\int_0^tF(u,X_0)du}f_1(X_{t_1})|X_0 \big) = \mathbb{E}e^{\int_0^tF(u,X_0)du}f_1(X_{t_1}).
\end{align*}
Assume that the assertion holds for $0<t_1<\ldots<t_{n-1}$ and $(n-1)$ positive, Borel functions. We conclude
from the Markov property that
\begin{align*}
	\mathbb{E}\Pi_{i=1}^nf_i(X_{t_i+t}) &= \mathbb{E}\mathbb{E}\Big(\Pi_{i=1}^nf_i(X_{t_i+t})|\mathcal{F}^X_{t_{n-1}+t}\Big) = \mathbb{E}\Big(\Pi_{i=1}^{n-1}f_i(X_{t_i+t})\mathbb{E}_{X_{t_{n-1}+t}}f_n(X_{t_n - t_{n-1}})\Big).
\end{align*}
From the induction assumption the last expression is equal to
\begin{align*}
	\mathbb{E}\Big(\Pi_{i=1}^{n-1}f_i(X_{t_i})\Big(\mathbb{E}_{X_{t_{n-1}}}f_n(X_{t_n - t_{n-1}})\Big)e^{\int_0^tF(u,X_0)du}\Big),
\end{align*}
which  from the Markov property is equal to
\begin{align*}
\mathbb{E}\Big(\Pi_{i=1}^{n-1}f_i(X_{t_i})\mathbb{E}\big (f_n(X_{t_n - t_{n-1}})|\mathcal{F}^X_{t_{n-1}}\big )e^{\int_0^tF(u,X_0)du}\Big) = \mathbb{E}\Big(\Pi_{i=1}^{n}f_i(X_{t_i})e^{\int_0^tF(u,X_0)du}\Big).
\end{align*}
This gives \eqref{j9-1}.
The proof   is complete.
\end{proof}

For an $\mathcal{F}^X_{\infty}$ measurable positive or bounded random variable $Z$ we
have the following representation theorem.

\begin{theorem}  \label{j-repr}
Assume that process  $X$ on $(\Omega,\mathcal{F},\mathbb{P})$ has the  FOED property and the inital  distribution with strictly positive density such that $\mathbb{E}|F(t,X_0)| < \infty$ for every $t\geq 0$. If $(X,\mathbb{P}_x)_{x\in E}$ is a Markov family with respect to $\mathbb{F}^X$, then
	for an $\mathcal{F}^X_{\infty}$-measurable positive or bounded random variable $Z$ we have
	\begin{align*}
		\mathbb{E} (Z\circ \theta_t) = \mathbb{E}Z + \mathbb{E}\int_0^t (Z\circ \theta_r) F(r,X_r) dr.
	\end{align*}
\end{theorem}
\begin{proof}
Using the assumption  $\mathbb{E}|F(t,X_0)| < \infty$  we conclude from Theorem \ref{m_FOD}  for  $r>0$ and
$\widetilde{Z}(r) = F(r,X_0)Z$ that
	\begin{align*}
		\mathbb{E}\big (\widetilde{Z}(r) \circ\theta_r\big ) = \mathbb{E}\widetilde{Z}(r)e^{\int_0^rF(X_0, u)du}.
	\end{align*}
Since $F(\cdot,y)$ is locally integrable, 	integrating both sides with respect to $r$ and using Fubini's theorem yield for any $t>0$
	\begin{align*}
		\mathbb{E}\int_0^t\widetilde{Z}(r) \circ\theta_r dr = \mathbb{E}\int_0^t\widetilde{Z}(r)e^{\int_0^rF(X_0, u)du}dr.
	\end{align*}
Hence, upon noting that $(\widetilde{Z}(r) \circ\theta_r)(\omega)=\widetilde{Z}(r)(\theta_r(\omega))=F(r,X_0(\theta_r(\omega)))Z(\theta_r(\omega))=(Z\circ\theta_r)(\omega) F(r,X_r)(\omega)$), we obtain
	\begin{align*}
		\mathbb{E}\int_0^t & (Z\circ\theta_r) F(r,X_r)dr  = \mathbb{E}\int_0^t F(r,X_0)Ze^{\int_0^r F(u,X_0)du}dr\\
		&= \mathbb{E}Z\Big[e^{\int_0^tF(u,X_0)du}-1\Big]=\mathbb{E}(Z \circ\theta_t) - \mathbb{E}Z,
	\end{align*}
	where in the last equality we used again Theorem \ref{m_FOD}, this time for $Z$.
\end{proof}

We end this section with a few examples of processes which are Markov and have the FOED property.

\begin{example} \label{GGex-j}
Let $B$ be the (real valued) standard Brownian motion process.

a)	(A Gauss-Gauss process with a parameter $a>0$) \\
Consider  $X_t = B_{t+a}$. Here we have $$\mu_0(dz)=\gamma(dz)=\frac{1}{\sqrt{2\pi a}}e^{-\frac{z^2}{2a}}dz\ \textrm{and}\  \mu_t(dz)=\frac{1}{\sqrt{2\pi (t+a)}}e^{-\frac{z^2}{2(t+a)}}dz. $$ Thus, using  Propositions \ref{Ex1} and \ref{prop2.4} we conclude that $X$ has the FOED property and the  function of
	evolution of distribution, say $F_a(u,z)$, satisfies
\begin{align}\label{GG}
	e^{\int_0^tF_a(u,z)du} = \sqrt{\frac{a}{a+t}}e^{\frac{z^2t}{2a(a+t)}}.
\end{align}
	 and thus it has the form
	\begin{align}\label{GFtz}
		F_a(t,z) = -\frac{1}{2(a+t)} + \Big(\frac{1}{2a} +\frac{1}{2(a+t)^2} \Big)z^2.
	\end{align}

b)
 Consider a model analogous to Gauss-Gauss, this time built on Ornstein-Uhlenbeck process. Namely, let $y\in\mathbb{R}$, $\lambda \geq 0$ and $Y$ be the unique strong solution of
	\begin{align*}
		Y_t = y + B_t - \lambda\int_0^tY_sds, \quad t\geq 0.
	\end{align*}
	For a brief characterization of such a process see \cite[App. I, point 24]{SB}. It is well known that
\[
 Y_t=ye^{-\lambda t}+ e^{-\lambda t}\int _0^te^{\lambda s}dB_s,\quad t\geq 0.
\]
Thus, $Y_t \sim N(\nu_t,\sigma^{2}_t)$, where $\nu_t=ye^{-\lambda t}$ and $\sigma^2_t=e^{-2\lambda t}\int_0^te^{2\lambda s}ds$. Now, let $X_t=Y_{t+a}$ for $a>0.$ Then,
$$\mu_0(dz)=\gamma(dz)=\frac{1}{\sqrt{2\pi \sigma^2_a}}e^{-\frac{(z-\nu_a)^2}{2\sigma^2_a}}dz\ \textrm{and}\  \mu_t(dz)=\frac{1}{\sqrt{2\pi \sigma^2_{t+a}}}e^{-\frac{(z-\nu_{t+a})^2}{2\sigma^2_{t+a}}}dz.$$
Consequently, using  Propositions \ref{Ex1} and \ref{prop2.4} we conclude that $X$ has the FOED property and the  function of
	evolution of distribution, say $F_a(u,z)$, satisfies
\begin{align}\label{OU}
	e^{\int_0^tF_a(u,z)du} = \sqrt{\frac{ \sigma^2_{a}}{ \sigma^2_{t+a}}}e^{-\frac{(z-\nu_{t+a})^2}{2\sigma^2_{t+a}}+\frac{(z-\nu_{a})^2}{2\sigma^2_{a}}}.
\end{align}
Clearly, with $y=\lambda =0$ the result \eqref{OU} agrees with \eqref{GG}. \\
In \cite[App. I, point 24]{SB} the transition density for $X$ is given relative to the speed measure of $X$. Accordingly, it can be shown that
\[\mu_t(dz)=\frac{1}{2}\frac{\sqrt{\lambda}}{\sqrt{2\pi \sinh(\lambda t)}} \exp\Big\{\frac{\lambda}{2}\Big[t + y^2 + z^2  -(y^2 + z^2)\coth(\lambda t) + \frac{2zy}{\sinh(\lambda t)}   \Big]\Big\}m(dz),\]
where $m(dz) = 2e^{-\lambda z^2}dz$ is the speed measure for $X$. Given that, the ''speed measure'' representation for the function of evolution of distribution takes the form
\begin{align*}
	e^{\int_0^tF(u,z)du} = \sqrt{\frac{\sinh(\lambda a)}{\sinh(\lambda (a+t))}}\exp\Big\{&\frac{\lambda}{2}\Big(t + (y^2+z^2)(\coth(\lambda a)-\coth(\lambda (a+t)))\\
	&+2yz\big(\frac{1}{\sinh(\lambda a)}- \frac{1}{\sinh(\lambda(a+t))}\big)\Big)\Big\}.
\end{align*}
The ''speed measure'' representation approach will be taken in part c) below as well.

c) Consider a model built on a squared Bessel process. Namely,  let $y\geq 0$, $\delta \geq 0$ and $Y$ be the unique strong solution of
\begin{align*}
	Y_t = y + 2\int_0^t\sqrt{Y_s}dB_s + \delta t, \quad t\geq 0.
\end{align*}
 For a brief characterization of such a process see \cite[App. I, point 23]{SB}. It is well known that
\begin{align*}
	\mu_t(dz) = \frac{1}{t} \big({yz}\big)^{\frac{1}{ 2}(1-\frac{\delta}{2})}\exp\Big(-\frac{y+z}{2t}\Big)I_{\frac{\delta}{2}-1}\Big(\frac{\sqrt{yz}}{t}\Big)m(dz),
\end{align*}
where $I_c$ denotes the modified Bessel function with index $c\in\mathbb{R}$. Then for fixed $a>0$ and $X_t = Y_{t+a}$, the process $X$ is Markov with the FOED property  and again by (\ref{abscon})
\begin{align*}
	e^{\int_0^tF(u,z)du} = \frac{a}{a+t}\exp\Big\{\frac{(y+z)t}{2a(a+t)}\Big\}\frac{I_{\frac{\delta}{2}-1}\Big(\frac{\sqrt{yz}}{t+a}\Big)}
{I_{\frac{\delta}{2}-1}\Big(\frac{\sqrt{yz}}{a}\Big)}.
\end{align*}

\end{example}

\section{Feller processes with the  FOED property}

{When we deal with Feller processes we will make standard assumption that $E$ is a LCCB space.} Let  $X$ is be a  c\'adl\'ag  Feller process (in the sense of \cite[Ch. III]{RY}) admitting function $F$ of evolution of distribution and with the infinite lifetime. By $\mathcal{C}_0(E)$ we denote  the space of bounced and continuous functions on $E$
vanishing at infinity, and by $\mathbb{D}_A\subset \mathcal{C}_0(E)$ we denote the domain of infinitesimal generator of $X$.
By $C^2_c(E)$ we denote  the space of continuous functions on $E$ with compact support and both first and second derivatives continuous.

 The next result shows how one-dimensional distributions of $X$ depend on $A$ and~$F.$
\begin{proposition} Let $f \in\mathbb{D}_A$. Then
\begin{align}\label{AF}
	\mathbb{E}f(X_t) = \mathbb{E}\Big(f(X_0) + Af(X_0)\int_0^te^{\int_0^sF(u,X_0)du}ds\Big), \quad t\geq 0.
\end{align}
\end{proposition}
\begin{proof}  Since $f\in\mathbb{D}_A$ the process $f(X_t) - f(X_0) - \int_0^tAf(X_s)ds$ is a martingale, so
\begin{align*}
	\mathbb{E}f(X_t) = \mathbb{E}\Big(f(X_0) + \int_0^tAf(X_s)ds\Big).
\end{align*}
Because $Af\in \mathcal{C}_0(E)$  we conclude from Fubini's theorem and the FOED property that
 \begin{align*}
	\mathbb{E}f(X_t) &= \mathbb{E}f(X_0) + \int_0^t\mathbb{E}Af(X_s)ds = \mathbb{E}f(X_0) + \int_0^t\mathbb{E}Af(X_0)e^{\int_0^s F(u,X_0) du}ds\\
	&= \mathbb{E}\Big(f(X_0) + Af(X_0)\int_0^te^{\int_0^s F(u,X_0)du}ds\Big).
\end{align*}
The proof is complete.
\end{proof}

 If $t\mapsto F(t,z) \in \mathcal{C}^n(\mathbb{R}_+)$ for each $z\in E$, then we define a sequence of functions  $\mathcal{L}^m F:[0,\infty)\times E \rightarrow \mathbb{R}$ by
\begin{align}\label{eq:rec-L}
	\mathcal{L}^0F\equiv 1, \ \mathcal{L}^1F\equiv F, \ \mathcal{L}^{m+1}F = \frac{d}{dt}\mathcal{L}^m F + F \mathcal{L}^m F, \quad m\leq n.
\end{align}
Next, define
\begin{align*}
	\mathbb{D}^{(n)}_A = \{f\in \mathbb{D}_A: A^if \in \mathbb{D}_A \  {\rm for} \ i \leq n\},
\end{align*}
 and $A^{n+1}f = (A\circ A^n)f$ for $f\in \mathbb{D}^{(n)}_A $, $n\in\mathbb{N}$, with $A^1=A$.

\begin{proposition}\label{idan} Fix $n\in\mathbb{N}$.
Assume that the function $t\mapsto F(t,z) \in \mathcal{C}^n(\mathbb{R}_+)$ for each $z\in E$ and $\mathbb{E}|F(t,X_t)|<\infty$ for each $t\geq 0$. Let $f\in\mathbb{D}^{(n)}_A$. Then
\begin{align}\label{AnLn}
\mathbb{E} (A^nf(X_t)) = \mathbb{E}\Big(f(X_t)\mathcal{L}^nF(t,X_t)\Big)
\end{align}
 for every $t>0$.
\end{proposition}

\begin{proof} The proof goes  by induction.
For $n=1$ using Theorem \ref{j-repr} for $Z = f(X_0)$ yields
\begin{align*}
	\mathbb{E}f(X_t) = \mathbb{E}f(X_0) + \mathbb{E}\int_0^tf(X_s)  F(s,X_s) ds.
\end{align*}
On the other hand we have for $f\in C^2_c(E)$ (see \cite[Remark p.295]{RY})
	\begin{align*}
	\mathbb{E}f(X_t) = \mathbb{E}f(X_0) + \mathbb{E}\int_0^tAf(X_s)ds.
	\end{align*}
Comparing both equalities, using Fubini's theorem and taking derivative with respect to $t$ yield the assertion.

Assume now that identity (\ref{AnLn})  is true for any $k\leq n -1$. We have
\begin{align} \label{j13-1}
	\mathbb{E}A^nf(X_t) &= \mathbb{E}A(A^{n-1}f(X_t)) =  \mathbb{E}F(t,X_t)A^{n-1}f(X_t) \notag
	\\
	&= \mathbb{E}F(t,X_0)e^{\int_0^tF(u,X_0)du}A^{n-1}f(X_0),
\end{align}
where in the last equality we used the assumption that $X$ has the  FOED property.

Next,   by assumption $\mathbb{E}|F(0,X_0)|<\infty$  and  $f\in\mathbb{D}^{(n)}_A$, so $z\mapsto A^{n-1}f(z)$ is bounded and we may differentiate with respect to $t$ under expectation to conclude that RHS of \eqref{j13-1} is equal to
 \begin{align*}
	\mathbb{E}F(t,X_0)e^{\int_0^tF(u,X_0)du}A^{n-1}f(X_0) = \frac{d}{dt}\Big(\mathbb{E}e^{\int_0^tF(u,X_0)du}A^{n-1}f(X_0)\Big).
\end{align*}
  By the  FOED property and the induction assumption we obtain
\begin{align*}
&\frac{d}{dt}\Big(\mathbb{E}e^{\int_0^tF(u,X_0)du}A^{n-1}f(X_0)\Big) = \frac{d}{dt}\Big(\mathbb{E}A^{n-1}f(X_t)\Big) = \frac{d}{dt}\mathbb{E}\Big(f(X_t)\mathcal{L}^{n-1}F(t,X_t)\Big) \\
& = \frac{d}{dt}\mathbb{E}\Big(f(X_0)e^{\int_0^tF(u,X_0)du}\mathcal{L}^{n-1}F(t,X_0)\Big) = \mathbb{E}\Big(f(X_0)\frac{d}{dt}\Big(e^{\int_0^tF(u,X_0)du}\mathcal{L}^{n-1}F(t,X_0)\Big)\Big),
\end{align*}
which is equal to
\begin{align*}
 \mathbb{E}\Big(f(X_0)\Big(F(t,X_0)e^{\int_0^tF(u,X_0)du}\mathcal{L}^{n-1}F(t,X_0) +e^{\int_0^tF(u,X_0)du}\frac{d}{dt}\mathcal{L}^{n-1}F(t,X_0) \Big)\Big).
\end{align*}
Thus, using the above as well as  \eqref{eq:rec-L} and the  FOED property, we obtain
\begin{align*}
\mathbb{E}&A^nf(X_t)\\
&=\mathbb{E}\Big(f(X_0)\Big(F(t,X_0)e^{\int_0^tF(u,X_0)du}\mathcal{L}^{n-1}F(t,X_0)+e^{\int_0^tF(u,X_0)du}\frac{d}{dt}\mathcal{L}^{n-1}F(t,X_0) \Big)\Big)\\
&=\mathbb{E}f(X_0)\mathcal{L}^{n}F(t,X_0)e^{\int_0^tF(u,X_0)du} = \mathbb{E}\Big(f(X_t)\mathcal{L}^nF(t,X_t)\Big).
\end{align*}
The proof is complete.
\end{proof}

\begin{example}\label{GGex}
	 Consider the Gauss-Gauss process from Example \ref{GGex-j} with parameter $a>0$. {We find a description of the functions $\mathcal{L}^m F$ in terms of the density of $X_t$.}

The assumptions of  Proposition \ref{idan} are satisfied.  Let $g_{t+a}$ denote the density of $B_{t+a}$. Using formula (\ref{AnLn}) for $f(x) = e^{-\lambda x}$ with $\lambda > 0$ and $X_t = B_{t+a}$,  we obtain
\begin{align*}
	\frac{\lambda^{2n}}{2^n}\mathbb{E}e^{-\lambda B_{t+a}} = \mathbb{E}e^{-\lambda B_{t+a}}\mathcal{L}^nF(t,B_{t+a}), \qquad t> 0.
\end{align*}
Let now $f(s)=1_{\{s>0\}}\frac{s^{2n-1}}{(2n-1)!}$ and $h(z)=\mathcal{L}^nF(t,z)g_{t+a}(z)$. Hence, since $\lambda^{-2n} = \int_0^{\infty}e^{-\lambda s}\frac{t^{2n-1}}{(2n-1)!}ds=\int_{-\infty}^{\infty}e^{-\lambda s}f(s)ds$, then, using the convolution theorem for  Laplace transforms, we obtain the following chain of equalities for each $\lambda > 0$,
\begin{align*}
		\mathbb{E}e^{-\lambda B_{t+a}} &= 	2^n	\lambda^{-2n}	\mathbb{E}e^{-\lambda B_{t+a}}\mathcal{L}^nF(t,B_{t+a})\\
		&= 2^n\Big(\int_{-\infty}^{\infty}e^{-\lambda s}f(s)ds\Big)\Big(\int_{-\infty}^{\infty}e^{-\lambda z}g(z)dz\Big)\\
		&= 2^n\int_{-\infty}^{\infty}e^{-\lambda v}\left(\int_{-\infty}^\infty f(v-z)h(z)dz\right )dv
\\
		&= 2^n\int_{-\infty}^{\infty}e^{-\lambda v}\left(\frac{1}{\Gamma(2n)}\int_{-\infty}^v(v-z)^{2n-1}g_{t+a}(z)\mathcal{L}^nF(t,z)dz\right )dv,
\end{align*}
so that we obtain an integral equation for $\mathcal{L}^nF$
\begin{align*}
	g_{t+a}(v) = \frac{2^n}{\Gamma(2n)}\int_{-\infty}^v(v-z)^{2n-1}g_{t+a}(z)\mathcal{L}^nF(t,z)dz.
\end{align*}
Using the Leibnitz rule iteratively we can compute $\mathcal{L}^nF(t,z)$. For example, for $n=1$ we get
\begin{align*}
\mathcal{L}^1F(t,v)=\frac{1}{2}\frac{g''_{t+a}(v)}{g_{t+a}(v)}.
\end{align*}
\qed
\end{example}

\section{Bridge operators}

In order to study the  finite dimensional distributions of a Markov process $X$ with the initial distribution $\gamma$ having $E$ as the support and  admitting the FOED property, we introduce a bridge operator.

Towards this end we consider a Banach space $L^1(\gamma):= L^1(E, \mathcal{B}(E), \gamma)$ with the norm  given by $\left\|h\right\| = \mathbb{E}|h(X_0)|$.
  For  $t\geq 0$ and for $f\in L^1(\gamma)$ we denote
	$$\Lambda_tf=\mathbb{E} (f(X_0)\big|X_t). $$
For each $t\geq 0$ the operator $f\mapsto \Lambda_tf$ is a linear contraction  on $L^1(\gamma)$ as well as on $L^2(\gamma)$.
We can treat 	$\Lambda_tf$ as a function on $E$, namely
$\Lambda_tf(z) =h(z), z \in E$, where	$h$ is a Borel function such that $\mathbb{E}\big(f(X_0)|X_t\big)=h(X_t)$,
so
 $$\Lambda_tf(z) = \mathbb{E} (f(X_0)\big|X_t=z). $$
 We will always assume that $(t,z)\mapsto \Lambda_t(f)(z)$ is a measurable function.
Observe also that for a  Markov process $X$ with the FOED property, with the initial distribution $\gamma(dz)=g(z)dz$ and with the transition density $p_t(\cdot,\cdot)$ we have a nice formula for  $\Lambda_tf$
	\begin{align}\label{Lmt}
		\Lambda_t f(z) =\frac{\int_Ef(v)p_t(v,z)\mu_0(dv)}{\mu_t(dz)/dz}= \frac{\int_Ef(v)p_t(v,z)\gamma(dv)}{g(z)e^{\int_0^tF(u,z)du}}, \quad f \in L^1(\gamma),
	\end{align}
	where for the second equality we used \eqref{abscon}. So in this case clearly  $(t,z)\mapsto\Lambda_t f(z) $ is measurable.
We will call  the operator $\Lambda_t$ the \textit{bridge operator}.

We will demonstrate below that the finite dimensional distributions of $X$ can be described in terms of the bridge operators $(\Lambda_t)_{t>0}$ and the FOED function. This gives a new algorithmic approach to the computational problems related to finite dimensional distributions of a Markov process. The algorithm uses the initial distribution $\gamma$ and the form of bridge operator given by  (\ref{Lmt}). The first step to achieve this goal is is the theorem below.

\begin{theorem} \label{mdim} Assume that $X$ is a Markov process having the FOED property. Let $n\in\mathbb{N}$ and $0<t_1<\ldots<t_n$. Assume that $f_0,f_1,\ldots,f_n$ are Borel and bounded functions on $E$. Then
\begin{align}\label{findist}
	&\mathbb{E}f_0(X_0)f_1(X_{t_1})\ldots f_n(X_{t_n}) = \mathbb{E}\Big(f_n(X_0)e^{\int_0^{t_n-t_{n-1}}F(u,X_0)du}\\
	&\quad \times \Lambda_{t_n-t_{n-1}}\Big(f_{n-1}(\cdot)e^{\int_0^{t_{n-1}-t_{n-2}}F(u,\cdot)du}\Lambda_{t_{n-1}-t_{n-2}}\Big(\cdots\Lambda_{t_1}f_0\Big)(\cdot)\Big)(X_0)\Big).\notag
\end{align}
\end{theorem}
\begin{proof}
Notice that each successive component of
the RHS of (\ref{findist}) is well defined since $f_0,\ldots,f_n$ are Borel  bounded functions on $E$.  Let $c$ be such that $\max_{i\leq n}\|f_i\|_\infty \leq c$.

The proof of (\ref{findist}) goes by induction.   For $n =1$, we see that she function $f_1(\cdot)e^{\int_0^{t_1}F(u,\cdot)du}\Lambda_{t_1}f_0(\cdot) \in L^1(\gamma)$ since
\begin{align*}
	\mathbb{E}\Big|f_1(X_0)e^{\int_0^{t_1}F(u,X_0)du}\Lambda_{t_1}f_0(X_0)\Big|\leq c^2 \mathbb{E}e^{\int_0^{t_1}F(u,X_0)du} = c^2,
\end{align*} and then using the definitions of $\Lambda_t$ and having the FOED property   we obtain
\begin{align*}
\mathbb{E}f_0(X_0)f_1(X_{t_1}) = \mathbb{E}f_1(X_{t_1})\Lambda_{t_1}f_0(X_{t_1}) = \mathbb{E}f_1(X_{0})e^{\int_0^{t_{1}}F(u,X_0)du}\Lambda_{t_1}f_0(X_{0}).
\end{align*}

For the induction step suppose that identity (\ref{findist}) holds for $n-1$ and for bounded Borel functions $f_0,\ldots,f_{n-1}$. Using the Markov property and the induction assumption we obtain the following chain of equalities
\begin{align*}
	&\mathbb{E}f_0(X_0)f_1(X_{t_1})\ldots f_n(X_{t_n}) = \mathbb{E}\Big(f_0(X_0)f_1(X_{t_1})\ldots \ \mathbb{E}_{X_{t_{n-1}}}f_n(X_{t_n-t_{n-1}})\Big)\\
	&= \mathbb{E}\Big(f_{n-1}(X_0) \mathbb{E}_{X_{0}}f_n(X_{t_n-t_{n-1}})e^{\int_0^{t_{n-1}-t_{n-2}}F(u,X_0)du}\Lambda_{t_{n-1}-t_{n-2}}\Big(\cdots\Lambda_{t_1}f_0\Big)(X_0)\Big)\\
	&= \mathbb{E}\Big(f_{n-1}(X_0) f_n(X_{t_n-t_{n-1}})e^{\int_0^{t_{n-1}-t_{n-2}}F(u,X_0)du}\Lambda_{t_{n-1}-t_{n-2}}\Big(\cdots\Lambda_{t_1}f_0\Big)(X_0)\Big)\\
	&= \mathbb{E}\Big(f_n(X_{t_n-t_{n-1}})\\
	&\qquad\times \mathbb{E}\Big(f_{n-1}(X_0) e^{\int_0^{t_{n-1}-t_{n-2}}F(u,X_0)du}\Lambda_{t_{n-1}-t_{n-2}}\Big(\cdots\Lambda_{t_1}f_0\Big)(X_0)\Big|X_{t_n-t_{n-1}}\Big)\Big)\\
	&= \mathbb{E}\Big(f_n(X_{t_n-t_{n-1}})\\
	&\qquad  \times \Lambda_{t_n-t_{n-1}}\Big(f_{n-1}(\cdot) e^{\int_0^{t_{n-1}-t_{n-2}}F(u,\cdot)du}\Lambda_{t_{n-1}-t_{n-2}}\Big(\cdots\Lambda_{t_1}f_0\Big)(\cdot)\Big)\Big(X_{t_n-t_{n-1}}\Big)\Big).
\end{align*}
By the FOED property the last expression is equal to
\begin{align*}
\mathbb{E}\Big(f_n(X_{0})&e^{\int_0^{t_{n}-t_{n-1}}F(u,X_0)du}\\
&\quad \times\Lambda_{t_n-t_{n-1}}\Big(f_{n-1}(\cdot) e^{\int_0^{t_{n-1}-t_{n-2}}F(u,\cdot)du}\Lambda_{t_{n-1}-t_{n-2}}\Big(\cdots\Lambda_{t_1}f_0\Big)(\cdot)\Big)\big(X_0\big)\Big),
\end{align*}
and the assertion follows.
\end{proof}

For a fixed $t>0$ denote
\begin{align*}
	\Ker \Lambda_t = \{h\in L^ 1(\gamma): \Lambda_th = 0 \quad \gamma-a.s.\}.
\end{align*}
 It follows from standard arguments of functional analysis that $\Ker \Lambda_t$ is  a closed subspace of $L^1(\gamma)$.
 Any  $f \in L^1(\gamma)$ decomposes
\begin{align}\label{odeco}
	f = \Lambda_tf + (f)^{\bot}_t, \qquad t>0,
\end{align}
where $(f)^{\bot}_t \in \Ker \Lambda_t$. 
 We are ready to formulate the result, which gives a simple recurrent formula for determining the distribution of a vector $(X_0, X_{t_1},\ldots,X_{t_n})$. The formula is strictly connected with the notion of kernel of a bridge operator. We will need the following definition: For a given non-negative integer $n$, $0= t_0<t_1<\ldots<t_n$ and for bounded Borel functions $f_1,\ldots, f_n$ define
\begin{align}\label{xin}
	\Xi^{(n)}_{[f_1,..,f_n]}(t_1,\ldots,t_n,z) &= \prod_{i = 1}^nf_i(z)e^{\int_0^{t_{i}-t_{i-1}}F(u,z)du}, \quad n\geq 1, \ z \in E.
\end{align}
We will also need the following class of functions
\begin{align}\label{classk}
	\mathcal{K} = \{g\in L^ 1(\gamma): \quad \alpha \in \Ker \Lambda_s \ {\rm for }\  {\rm  all}\ s >0 \Rightarrow    \  g\alpha \in \Ker \Lambda_t  \ {\rm for }\  {\rm  all}\ t >0 \}.
\end{align}
Actually, for a broad class of diffusions we may expect that $\mathcal{K} = L^ 1(\gamma)$.
\begin{example} \label{Gauss-Gauss} Consider a Gauss-Gauss process, that is $X_t = B_{t+a}$, where $B$ is a standard Brownian motion and $a>0$ is a fixed number, so $\gamma(dz) = \frac{1}{\sqrt{2\pi a}}e^{-\frac{z^2}{2a}}dz$. Then  $\mathcal{K} = L^1(\gamma)$. Indeed, first observe that if $h \in L^1(\gamma)$  and
	\begin{align}\label{eq:h0}
		\mathbb{E}\big(h(B_a)\big|B_{t+a} = z\big) = 0,
	\end{align}
for each $z\in\mathbb{R}$, then
\begin{align*}
		\mathbb{E}\big(h(B_a)\big|B_{t+a}\big) = 0.
	\end{align*}
It follows that for every $\lambda\in \mathbb{R}$ we have
\begin{align*}
		0&=\mathbb{E}\left(\mathbb{E}\big(h(B_a)\big|B_{t+a}\big)e^{\lambda B_{t+a}}\right ) = \mathbb{E}\left(h(B_a)e^{\lambda B_{t+a}}\right )\\
&= \mathbb{E}\left(h(B_a)e^{\lambda (B_{t+a}-B_a)}e^{\lambda B_a}\right )=\mathbb{E}\left(e^{\lambda (B_{t+a}-B_a)}\right )\mathbb{E}\left(h(B_a)e^{\lambda B_a}\right ).
	\end{align*}
Consequently,
\begin{align*}
		\mathbb{E}\left(h(B_a)e^{\lambda B_a}\right )=0
	\end{align*}
for every $\lambda\in \mathbb{R}$, and so $h\equiv 0$. This yields that $\Ker\Lambda_t=\{0\}$, so $\mathcal{K} = L^1(\gamma)$. \\
 By  a similar arguments we may conclude analogous results for an Ornstein-Uhlenbeck process, a  Bessel process and many other processes. We omit the details.
\end{example}

{Using  the notion of class $\mathcal{K}$ we are able to give a description of $\mathbb{E}\prod_{i = 1}^nf_i(X_{t_i})$ in terms of distribution of $X_0$, which allows to compute the distribution of vector $(X_{t_1},\ldots,X_{t_n})$ for $X$ having the FOED property in terms of distribution of $X_0$.}

\begin{theorem} \label{RCR} Assume that $X$ is a Markov process having the FOED property. Let $n\in\mathbb{N}$ and $0=t_0<t_1<\ldots<t_n$. Assume that $f_1,\ldots ,f_n$ are bounded Borel functions. { Then
	\begin{align*}
		\mathbb{E}f_1(X_{t_1}) = \mathbb{E}\Xi^{(1)}_{[f_1]}(t_1,X_0).
	\end{align*}
Moreover, if for $n\geq 2$ and   $j\in\{1,\ldots,n-1\}$ it holds
	$$f_{n-1}(\cdot)\cdots f_j(\cdot)e^{\sum_{i=j}^{n-1}\int_0^{t_{i}-t_{i-1}}F(u,\cdot)du}\in\mathcal{K},$$
then	
  for $n>1$ we have
\begin{align}\label{recaR}
	\mathbb{E}\prod_{i = 1}^nf_i(X_{t_i}) = \mathbb{E}&\Big[\Xi^{(n)}_{[f_1,\ldots,f_n]}(t_1,\ldots,t_n,X_0)\\
	& - f_n(X_0)e^{\int_0^{t_n-t_{n-1}}F(u,X_0)du}\big(\Xi^{(n-1)}_{[f_1,\ldots,f_{n-1}]}(t_1,\ldots,t_{n-1},\cdot)\big)^{\bot}_{t_n - t_{n-1}}(X_0)\Big] \notag.
\end{align}}
\end{theorem}

\begin{proof}

\textit{Step 1.} Let $n = 1$. By definition, for $t_1>0$ and $z\in E$
\begin{align*}
 	\Xi^{(1)}_{[f_1]}(t_1,z)= f_1(z)e^{\int_0^{t_1}F(u,z)du}.
\end{align*}	
Hence
\begin{align*}
 \mathbb{E}\Big[f_1(X_0)e^{\int_0^{t_1}F(u,X_0)du}\Big] = \mathbb{E}f_1(X_{t_1}),
\end{align*}
where in that last equality we used  the definition of FOED. This proves the assertion for $n=1$.

\ \\
 \textit{Step 2.}
We recall formula (\ref{findist}) for $f_0\equiv 1$
\begin{align}\label{findistR}
	\mathbb{E}f_1(X_{t_1})&\ldots f_n(X_{t_n})	= \mathbb{E}\Big(f_n(X_0)e^{\int_0^{t_n-t_{n-1}}F(u,X_0)du}\\
	&\times \Lambda_{t_n-t_{n-1}}\Big(f_{n-1}(\cdot)e^{\int_0^{t_{n-1}-t_{n-2}}F(u,\cdot)du}\Lambda_{t_{n-1}-t_{n-2}}\Big(\cdots\Lambda_{t_1}1\Big)(\cdot)\Big)(X_0)\Big).\notag
\end{align}
 Take $n\geq 2$. Set
\begin{align*}
	\widetilde{\Lambda}_{1}(z) &= \Lambda_{t_1}1(z);\\
	\widetilde{\Lambda}_{i}(z) &= \Lambda_{t_{i}-t_{i-1}}\Big(f_{i-1}(\cdot)e^{\int_0^{t_{i-1}-t_{i-2}}F(u,\cdot)du}\Lambda_{t_{i-1}-t_{i-2}}\Big(\cdots\Lambda_{t_1}1\Big)(\cdot)\Big)(z), \quad i \geq 2, \ z\in E,
\end{align*}
so,  using (\ref{findistR}), we may write
\begin{align*}
\mathbb{E}f_1(X_{t_1})\ldots f_n(X_{t_n}) = \mathbb{E}\Big(&f_n(X_0)e^{\int_0^{t_n-t_{n-1}}F(u,X_0)du}\\
	&\times \Lambda_{t_n-t_{n-1}}\Big(f_{n-1}(\cdot)e^{\int_0^{t_{n-1}-t_{n-2}}F(u,\cdot)du}\widetilde{\Lambda}_{n-1}(\cdot)\Big)(X_0)\Big),
\end{align*}
which is equal to
\begin{align*}
	\mathbb{E}&\Big(f_n(X_0)e^{\int_0^{t_n-t_{n-1}}F(u,X_0)du}\\
	&\times \Lambda_{t_n-t_{n-1}}\Big(f_{n-1}(\cdot)e^{\int_0^{t_{n-1}-t_{n-2}}F(u,\cdot)du}\Lambda_{t_{n-1}-t_{n-2}}\Big(f_{n-2}
(\cdot)e^{\int_0^{t_{n-2}-t_{n-3}}F(u,\cdot)du}\widetilde{\Lambda}_{n-2}(\cdot)\Big)(X_0)\Big).
\end{align*}
We rewrite the last expression using decomposition (\ref{odeco}) as
\begin{align*}
\mathbb{E}&\Big(f_n(X_0)e^{\int_0^{t_n-t_{n-1}}F(u,X_0)du}\\
	&\times \Lambda_{t_n-t_{n-1}}\Big(f_{n-1}(\cdot)e^{\int_0^{t_{n-1}-t_{n-2}}F(u,\cdot)du}\Big(f_{n-2}(\cdot)e^{\int_0^{t_{n-2}-t_{n-3}}F(u,\cdot)du}\widetilde{\Lambda}_{n-2}(\cdot)\\
	& \quad - \Big(f_{n-2}(\cdot)e^{\int_0^{t_{n-2}-t_{n-3}}F(u,\cdot)du}\widetilde{\Lambda}_{n-2}(\cdot)\Big)^{\bot}_{t_{n-1}-t_{n-2}}\Big)(X_0)\Big) =:\mathcal{I}.
\end{align*}

By assumption $f_{n-1}(\cdot)e^{\int_0^{t_{n-1}-t_{n-2}}F(u,\cdot)du}\in\mathcal{K}$, so we obtain
\begin{align*}
	\mathcal{I} &= \mathbb{E}\Big(f_n(X_0)e^{\int_0^{t_n-t_{n-1}}F(u,X_0)du}\\
	&\times \Lambda_{t_n-t_{n-1}}\Big(f_{n-1}(\cdot)e^{\int_0^{t_{n-1}-t_{n-2}}F(u,\cdot)du}f_{n-2}(\cdot)e^{\int_0^{t_{n-2}-t_{n-3}}F(u,\cdot)du}\widetilde{\Lambda}_{n-2}(\cdot)\Big)(X_0)\Big),
\end{align*}
and, after iterating in the analogous way, we conclude that
\begin{align*}
&\mathbb{E}f_1(X_{t_1})\ldots f_n(X_{t_n})\\
&=	\mathbb{E}\Big(f_n(X_0)e^{\int_0^{t_n-t_{n-1}}F(u,X_0)du} \Lambda_{t_n-t_{n-1}}\Big(f_{n-1}(\cdot)f_{n-2}(\cdot)\cdots f_{1}(\cdot)e^{\sum_{i=1}^{n-1}\int_0^{t_{i}-t_{i-1}}F(u,\cdot)du}\widetilde{\Lambda}_{1}(\cdot)\Big)(X_0)\Big)\\
	&= \mathbb{E}\Big(f_n(X_0)e^{\int_0^{t_n-t_{n-1}}F(u,X_0)du} \Lambda_{t_n-t_{n-1}}\Big(f_{n-1}(\cdot)f_{n-2}(\cdot)\cdots f_{1}(\cdot)e^{\sum_{i=1}^{n-1}
	\int_0^{t_{i}-t_{i-1}}F(u,\cdot)du}\Lambda_{t_1}1(\cdot)\Big)(X_0)\Big)
\end{align*}
Since $\Lambda_{t_1}1=1$ we have
\begin{align*}
&\mathbb{E}f_1(X_{t_1})\ldots f_n(X_{t_n})\\
 &= \mathbb{E}\Big(f_n(X_0)e^{\int_0^{t_n-t_{n-1}}F(u,X_0)du} \Lambda_{t_n-t_{n-1}}\Big(f_{n-1}(\cdot)f_{n-2}(\cdot)\cdots f_{1}(\cdot)e^{\sum_{i=1}^{n-1}\int_0^{t_{i}-t_{i-1}}F(u,\cdot)du}\Big)(X_0)\Big)\\
 &= \mathbb{E}\Big(f_n(X_0)e^{\int_0^{t_n-t_{n-1}}F(u,X_0)du} \Big(f_{n-1}(\cdot)f_{n-2}(\cdot)\cdots f_{1}(\cdot)e^{\sum_{i=1}^{n-1}\int_0^{t_{i}-t_{i-1}}F(u,\cdot)du}\\
 &\quad - \big(f_{n-1}(\cdot)f_{n-2}(\cdot)\cdots f_{1}(\cdot)e^{\sum_{i=1}^{n-1}\int_0^{t_{i}-t_{i-1}}F(u,\cdot)du}\big)^{\bot}_{t_n-t_{n-1}}\Big)(X_0)\Big)\\
&= \mathbb{E}\Big(e^{\sum_{i=1}^{n}\int_0^{t_{i}-t_{i-1}}F(u,\cdot)du}\prod_{i = 0}^nf_i(X_0)\\
& \quad - f_n(X_0)e^{\int_0^{t_n-t_{n-1}}F(u,X_0)du}\big(f_{n-1}(\cdot)f_{n-2}(\cdot)\cdots f_{1}(\cdot)e^{\sum_{i=1}^{n-1}\int_0^{t_{i}-t_{i-1}}F(u,\cdot)du}\big)^{\bot}_{t_n-t_{n-1}}(X_0)\Big)\\
&= \mathbb{E}\Big[\Xi^{(n)}_{[f_1,\ldots,f_n]}(t_1,\ldots,t_n,X_0)\\
  & \quad - f_n(X_0)e^{\int_0^{t_n-t_{n-1}}F(u,X_0)du}\big(\Xi^{(n-1)}_{[f_1,\ldots,f_{n-1}]}(t_1,\ldots,t_{n-1},\cdot)\big)^{\bot}_{t_n - t_{n-1}}(X_0)\Big].
\end{align*}
The proof is complete.
\end{proof}

\begin{remark}
Let us comment on a possible use of formula (\ref{recaR}). Suppose that we know the distribution $\gamma$ of $X_0$ and the transition density $p_t(\cdot,\cdot)$ of a Markov process $X$. In the classical forward approach, in order to compute $\mathbb{E}\big(\prod_{i = 1}^nf_i(X_{t_i})\big)$ we need to compute a multiple integral involving a chain of transition densities. The alternative backward approach is completely different and uses the notion of bridge operator and its kernel. We have the following algorithm:

\textit{Step 1.} We determine $e^{\int_0^tF(u,z)du}$ as
\begin{align*}
	e^{\int_0^tF(u,z)du} = \frac{\mu_t(dz)}{\gamma(dz)}, \quad \mu_t(dz) = \int_Ep_t(v,dz)\gamma(dv).
\end{align*}
This is needed for the computation of the family of objects  $(\Xi^{(i)}_{[f_1,\ldots,f_i]}(t,z))_{t>0, i\leq n}$.

\textit{Step 2.} Using (\ref{xin}) we compute $(\Xi^{(i)}_{[f_1,\ldots,f_i]}(t_0,\ldots,t_i,z))_{0<t_1< ..< t_i \leq t, i\leq n}$ in straightforward recursive way.

\textit{Step 3.} To determine the distribution of $(X_0,X_{t_1},\ldots,X_{t_n})$ by applying Theorem \ref{RCR} we need to verify that for $n\geq 2$ and   $j\in\{1,\ldots,n-1\}$ it holds
	$$f_{n-1}(\cdot)\cdots f_j(\cdot)e^{\sum_{i=j}^{n-1}\int_0^{t_{i}-t_{i-1}}F(u,\cdot)du}\in\mathcal{K}.$$
	Once this is verified, we then follow the recurrence (\ref{recaR}) using formulas (\ref{Lmt}) and  (\ref{odeco}).
\end{remark}

 In the  corollary below we see that under appropriate assumptions   the FOED property reduces the complexity of computational problems related to finite-dimensional distributions of a Markov process. The standard approach relies on computations based on chain rule of transition densities. The FOED backward approach offers the new possibility of dealing with such   computations.
\begin{corollary} Under assumptions of Theorem \ref{RCR} and under the assumption that $X_0$ has a density $g>0$,   formula (\ref{recaR}) can be written as
\begin{align}\label{redu}
	\mathbb{E}\prod_{i = 1}^nf_i(X_{t_i}) = \mathbb{E}\Big[\frac{f_n(X_0)}{g(X_0)}\Xi^{(n-1)}_{[f_1,\ldots,f_{n-1}]}(t_1,\ldots,t_n,\widehat{X}_0)p_{t_n-t_{n-1}}(\widehat{X}_0, X_0)\Big],
\end{align}
where $\widehat{X}_0$ is an independent copy of $X_0$.
\end{corollary}
\begin{proof} Using (\ref{odeco}) we rewrite (\ref{recaR})  in the form
\begin{align*}
	\mathbb{E}&\prod_{i = 1}^nf_i(X_{t_i}) = \mathbb{E}\Big[\Xi^{(n)}_{[f_1,\ldots,f_n]}(t_1,\ldots,t_n,X_0)  - f_n(X_0)e^{\int_0^{t_n-t_{n-1}}F(u,X_0)du}\\
	&\times \Big(\Xi^{(n-1)}_{[f_1,\ldots,f_{n-1}]}(t_1,\ldots,t_{n-1},X_0) - \Lambda_{t_n - t_{n-1}}(\Xi^{(n-1)}_{[f_1,\ldots,f_{n-1}]}(t_1,\ldots,t_{n-1},\cdot))(X_0)\Big)\Big]\\
	&= \mathbb{E}\Big(f_n(X_0)e^{\int_0^{t_n-t_{n-1}}F(u,X_0)du}\Lambda_{t_n - t_{n-1}}(\Xi^{(n-1)}_{[f_1,\ldots,f_{n-1}]}(t_1,\ldots,t_{n-1},\cdot))(X_0)\Big).
\end{align*}
By (\ref{Lmt}) the last expression is  equal to
\begin{align*}
\mathbb{E}&\Big(\frac{f_n(X_0)}{g(X_0)}\int_E(\Xi^{(n-1)}_{[f_1,\ldots,f_{n-1}]}(t_1,\ldots,t_{n-1},v)p_{t_n - t_{n-1}}(v,X_0)g(v)dv)\Big)\\
&=\mathbb{E}\Big(\frac{f_n(X_0)}{g(X_0)}\mathbb{E} \left(\Xi^{(n-1)}_{[f_1,\ldots,f_{n-1}]}(t_1,\ldots,t_{n-1},\widehat{X}_0)p_{t_n - t_{n-1}}(\widehat{X}_0,X_0)|X_0\right )\Big)\\
&= \mathbb{E}\Big(\frac{f_n(X_0)}{g(X_0)}\Xi^{(n-1)}_{[f_1,\ldots,f_{n-1}]}(t_1,\ldots,t_{n-1},\widehat{X}_0)p_{t_n - t_{n-1}}(\widehat{X}_0,X_0)\Big),
\end{align*}
 where the first equality is a consequence of the independence lemma (see e.g. \cite{Mik}, page 72, Rule 7). This finishes the proof.
\end{proof}

 Using Theorem \ref{RCR} we are able to determine the finite-dimensional distributions of the  process $e^X$ for various processes $X$ satisfying  the FOED property.
\begin{example} Let $X$ be a Gauss-Gauss or an Ornstein-Uhnlenbeck process. We saw in Example \ref{GGex-j} that $e^{\int_0^{t}F(u,z)du}$ is a function of order $e^{c z^2}$ with respect to $z$ for a constant $c>0$ dependent on $t $.    Since, as shown in Example \ref{Gauss-Gauss},  $\mathcal{K} = L^1(\gamma)$, then to apply Theorem \ref{RCR} we need to verify integrability of the functions
$$f_{n-1}(\cdot)\cdots f_j(\cdot)e^{\sum_{i=j}^{n-1}\int_0^{t_{i}-t_{i-1}}F(u,\cdot)du}$$
for $j\leq n-1$.
This will always hold if we choose $f_i(z) = e^{-\lambda_i e^z}$ for $\lambda_i > 0$, $i =1,..,n$. In this case formula (\ref{redu}) enables us to establish a formula for Laplace transform for vector of exponents of $X_{t_i}$, $i =1,..,n$, namely
\begin{align*}
	\mathbb{E}e^{-\sum_{i=1}^n\lambda_ie^{X_{t_i}}} = \mathbb{E}\Big(e^{-\lambda_ne^{X_0} -\sum_{i=1}^n\lambda_ie^{\widetilde{X}_{0}}}e^{ \sum_{i=1}^{n-1}\int_0^{t_i - t_{i-1}}F(u,\widetilde{X}_0)du} \frac{p_{t_n - t_{n-1}}(\widetilde{X}_0,X_0)}{g(X_0)} \Big),
\end{align*}
where $\widetilde{X}_0$ is an independent copy of $X_0$. In terms of computation the above equality provides a reduction of a multivariate integration to  bivariate integration.
\end{example}
Under additional assumptions we   are able to describe   the distribution of vector $(X_1,\ldots,X_n)$ in terms of distribution of $X_0$.
\begin{corollary}  \label{dist}
	Assume that $X$ is a Markov process with the FOED property,  $\mathcal{K} = L^1(\gamma)$   and
$e^{\sum_{i=1}^{n-1}\int_0^{t_{i}-t_{i-1}}F(u,\cdot)du}\in L^1(\gamma)$ for $0<t_1<\ldots<t_n$. Then for bounded Borel functions $f_1,\ldots, f_n$   we have
\begin{align}\label{xindi}
	\mathbb{E}\prod_{i = 1}^nf_i(X_{t_i}) = \mathbb{E}\Xi^{(n)}_{[f_1,\ldots,f_n]}(t_1,\ldots,t_n,X_0).
\end{align}
\end{corollary}
\begin{proof} If $\mathcal{K} = L^1(\gamma)$, then $\Ker \Lambda_t = \{0\}$ for each $t>0$. Indeed, take $\alpha \in \Ker \Lambda_t$.
 Since $\alpha \in  L^1(\gamma)=\mathcal{K}$, by definition of $\mathcal{K}$ it follows that $\alpha^2 \in \Ker \Lambda_t$  which easily yields that $\alpha\equiv 0$. The assertion follows from Theorem \ref{RCR}.
 \end{proof}

Then formula (\ref{xindi}) becomes powerful. Let us see this on the following example.
\begin{example} Consider a Gauss-Gauss process $(X_t)_{t\geq 0} = (B_{t+a})_{t\geq 0}$ for a fixed $a>0$. Hence
\begin{align*}
	\gamma(dz) = \frac{1}{\sqrt{2\pi a}}e^{-\frac{z^2}{2a}}, \qquad e^{\int_0^tF(u,z)du} = \sqrt{\frac{a}{a+t}}e^{\frac{z^2t}{2a(a+t)}}.
\end{align*}
Let for a given natural $n$, $t_i = i \leq n$ and let $f_1,\ldots, f_n$ be bounded Borel functions. Then we easily find that the condition
$e^{\sum_{i=1}^{n-1}\int_0^{t_{i}-t_{i-1}}F(u,\cdot)du}\in L^1(\gamma)$ holds if $n\leq a + 1$. Since $\mathcal{K} = L^1(\gamma)$ the formula (\ref{xindi}) is true in that case.
\end{example}


\begin{example} Let $\gamma(dx) = g(x)dx$ with $g>0$. Suppose that $X$ is a Markov process  with the transition density $p$
 such that $t\mapsto p_t(x,y)\in\mathcal{C}^1[0,\infty)$ for $x,y \in E$.
 Then by Proposition \ref{Ex1} the process $X$ has the FOED property and
\begin{align*}
	e^{\int_0^tF(u,x)du} = \frac{\int_Ep_t(z,x)g(z)dz}{g(x)}\leq \mathbb{E}k(t,X_0),
\end{align*}
where $k(t,z) =\sup_{x}	\frac{p_t(z,x)}{g(x)},\ t\geq 0, \ z\in E.$ If, moreover, $\mathbb{E}k(t_i-t_{i-1},X_0)<\infty$ for each $i$,
then
\begin{align*}
	 e^{\sum_{i=1}^{n}\int_0^{t_i-t_{i-1}}F(u,z)du}\leq \prod_{i=1}^n\mathbb{E}k(t_i-t_{i-1},X_0),
\end{align*}
  which means that $e^{\sum_{i=1}^{n-1}\int_0^{t_{i}-t_{i-1}}F(u,\cdot)du}\in L^1(\gamma)$. If additionally $\mathcal{K} = L^1(\gamma)$, assumptions of Corrolary \ref{dist} are satisfied.
\end{example}

\section{Conditional structures and FOED}

 One of the consequences of the FOED property and of the definition of bridge operators is a convenient computation of conditional structures of Markov processes. By a conditional structure of a Markov process we mean a function of the form
\begin{align}\label{cs-Psi}
	\mathbb{E}\Big(\Psi(X_0,X_{t_1})\cdots,X_{t_n})\Big| X_{s + t_1 + \ldots t_n} = z \Big), \quad z \in E,
\end{align}
for an integrable Borel function  $\Psi:E^{n+1}\rightarrow \mathbb{R}$, $0 < t_1 < \ldots  < t_n$ and $s\geq 0$. A particular  conditional structure takes the form
\begin{align}\label{cs}
	\mathbb{E}\Big(f_0(X_0)f_1(X_{t_1})\cdots f_n(X_{t_n})\Big| X_{s + t_1 + \ldots t_n} = z \Big), \quad z \in E,
\end{align}
for a sequence $(f_i)_{i=0,\ldots,n}$ of bounded Borel functions, $0 < t_1 < \ldots  < t_n$ and $s\geq 0$.
Such conditioning is motivated by use of the FOED property in what follows.
The conditional structures and theory of stochastic bridges are strictly connected.
For  a Markov process with the FOED property using the bridge operators $(\Lambda_{t})_{t\geq 0}$ we can give a full characterization of objects given by (\ref{cs}). For $s\geq 0$ and $0 < t_1 < \ldots  < t_n$ let
\begin{align*}
	T^s_{k,n} = s+\sum_{i=1}^{k}t_i,
\end{align*}
and let $\mu_t$ be the distribution of $X_t$.

\begin{theorem}\label{mltifo} Assume that $X$ is a Markov process having the FOED property and the initial distribution $\gamma$. Let $n\in\mathbb{N}$, $s\geq 0$ and $0 < t_1<\ldots<t_n$. Assume that $f_0,\ldots ,f_n$ are bounded Borel functions. Then a.e. with respect to the Lebesgue measure
\begin{align}\label{csfd}
\mathbb{E}\Big(f_0(X_0)&f_1(X_{t_1})\cdots f_n(X_{t_n})\Big| X_{T^s_{n,n}} = z \Big)
 = e^{-\int_{T^s_{n-1,n}}^{T^s_{n,n}}F(u,z)du}\notag\\
&\quad\times\Lambda_{T^s_{n-1,n}}\Big(f_n(\cdot)e^{\int_0^{t_n-t_{n-1}}F(u,\cdot)du}\Lambda_{t_n - t_{n-1}}\Big(\ldots\Lambda_{t_1}f_0\Big)(\cdot)\Big)(z).
\end{align}
\end{theorem}
\begin{proof} Let $B\in\mathcal{E}$. By the Markov property and Theorem \ref{mdim} we have
\begin{align*}
\mathbb{E}&\Big(f_0(X_0)f_1(X_{t_1})\cdots f_n(X_{t_n})1_{\{X_{T^s_{n,n}}\in B\}}\Big)\\
&= \mathbb{E}\Big(f_0(X_0)f_1(X_{t_1})\cdots f_n(X_{t_n})\mathbb{P}_{X_{t_n}}\big(X_{T^s_{n-1,n}}\in B\big)\Big)\\
	&= \mathbb{E}\Big(f_n(X_0)e^{\int_0^{t_n-t_{n-1}}F(u,X_0)du}\mathbb{P}_{X_{0}}\big(X_{T^s_{n-1,n}}\in B\big)\\
	&\qquad \times \Lambda_{t_n-t_{n-1}}\big(f_{n-1}(\cdot)e^{\int_0^{t_{n-1}-t_{n-2}}F(u,\cdot)du}\Lambda_{t_{n-1}-t_{n-2}}\big(\cdots \Lambda_{t_1}f_0\big)\big)(X_0)\Big)\\
	&= \mathbb{E}\Big(f_n(X_0)e^{\int_0^{t_n-t_{n-1}}F(u,X_0)du}\mathbb{E}\big(1_{\{X_{T^s_{n-1,n}}\in B\}}\big|\mathcal{F}^X_0\big)\\
	&\qquad \times \Lambda_{t_n-t_{n-1}}\big(f_{n-1}(\cdot)e^{\int_0^{t_{n-1}-t_{n-2}}F(u,\cdot)du}\Lambda_{t_{n-1}-t_{n-2}}\big(\cdots \Lambda_{t_1}f_0\big)\big)(X_0)\Big)\\
	&= \mathbb{E}\Big(1_{\{X_{T^s_{n-1,n}}\in B\}}f_n(X_0)e^{\int_0^{t_n-t_{n-1}}F(u,X_0)du}\\
	&\qquad \times \Lambda_{t_n-t_{n-1}}\big(f_{n-1}(\cdot)e^{\int_0^{t_{n-1}-t_{n-2}}F(u,\cdot)du}\Lambda_{t_{n-1}-t_{n-2}}\big(\cdots \Lambda_{t_1}f_0\big)\big)(X_0)\Big).
	\end{align*}
		Rewriting the last expression by conditioning with respect to $X_{s + t_1 + \ldots + t_{n-1}}$ we obtain
	\allowdisplaybreaks
	\begin{align*}
	 \mathbb{E}&\Big(\mathbb{E}\big(f_n(X_0)e^{\int_0^{t_n-t_{n-1}}F(u,X_0)du}\Lambda_{t_n-t_{n-1}}\big(f_{n-1}(\cdot)e^{\int_0^{t_{n-1}-t_{n-2}}F(u,\cdot)du}\\
	&\quad \times \ \Lambda_{t_{n-1}-t_{n-2}}\big(\cdots \Lambda_{t_1}f_0\big)\big)(X_0)\big|X_{T^s_{n-1,n}}\big) \ 1_{\{X_{T^s_{n-1,n}}\in B\}}\Big)\\
	&= \int_B\Lambda_{T^s_{n-1,n}}\Big(f_n(X_0)e^{\int_0^{t_n-t_{n-1}}F(u,X_0)du}\\
	&\quad \times \Lambda_{t_n-t_{n-1}}\big(f_{n-1}(\cdot)e^{\int_0^{t_{n-1}-t_{n-2}}F(u,\cdot)du}\Lambda_{t_{n-1}-t_{n-2}}\big(\cdots \Lambda_{t_1}f_0\big)\big)(z) \Big)
	   \mathbb{P}\big(X_{T^s_{n-1,n}}\in dz\big).
\end{align*}
Denoting
\begin{align*}
	R_{s,t_1,\ldots,t_{n}}(z) &= \Lambda_{T^s_{n-1,n}}\big(f_n(X_0)e^{\int_0^{t_n-t_{n-1}}F(u,X_0)du}\\
	& \times \Lambda_{t_n-t_{n-1}}\big(f_{n-1}(\cdot)e^{\int_0^{t_{n-1}-t_{n-2}}F(u,\cdot)du}\Lambda_{t_{n-1}-t_{n-2}}\big(\cdots \Lambda_{t_1}f_0\big)\big)(z),
\end{align*}
we conclude after comparing the first and the last term of the above sequence of equalities that
\begin{align}\label{R}
\mathbb{E}\big(f_1(X_{t_1})\cdots f_n(X_{t_n})&1_{\{X_{T^s_{n,n}}\in B\}}\big) = \int_BR_{s,t_1,\ldots,t_{n}}(z)\mu_{T^s_{n-1,n}}(dz).
\end{align}
On the other hand
\begin{align}\label{S1}
\mathbb{E}\big(f_1(X_{t_1})&\cdots f_n(X_{t_n})1_{\{X_{T^s_{n,n}}\in B\}}\big)\notag\\
	&= \mathbb{E}\big(\mathbb{E}\big(f_1(X_{t_1})\cdots f_n(X_{t_n})\big|X_{T^s_{n,n}}\big)1_{\{X_{T^s_{n,n}}\in B\}}\big),
\end{align}
so denoting
\begin{align*}
	S_{s,t_1,\ldots,t_{n}}(z) = \mathbb{E}\big(f_1(X_{t_1})&\cdots f_n(X_{t_n})\big|X_{T^s_{n,n}} = z\big), \quad z\in E,
\end{align*}
we can write  (\ref{S1}) in the form
\begin{align*}
\mathbb{E}\Big(f_1(X_{t_1})\cdots f_n(X_{t_n})&1_{\{X_{T^s_{n,n}}\in B\}}\Big) = \int_BS_{s,t_1,\ldots,t_{n}}(z)\mu_{T^s_{n,n}}(dz).
\end{align*}
Compare this and (\ref{R}) 
 to obtain
\begin{align*}
\int_BR_{s,t_1,\ldots,t_{n}}(z)\mu_{T^s_{n-1,n}}(dz) = \int_BS_{s,t_1,\ldots,t_{n}}(z)\mu_{T^s_{n,n}}(dz),
\end{align*}
which by the FOED property is equivalent to
\begin{align*}
\int_BR_{s,t_1,\ldots,t_{n}}(z)\mu_{T^s_{n-1,n}}(dz) = \int_BS_{s,t_1,\ldots,t_{n}}(z)e^{\int_{T^s_{n-1,n}}^{{T^s_{n,n}}}F(u,z)du}\mu_{T^s_{n-1,n}}(dz).
\end{align*}
The arbitrary choice of $B$ yields $z$- a.e.
\begin{align*}
	S_{s,t_1,\ldots,t_{n}}(z) = e^{-\int_{T^s_{n-1,n}}^{{T^s_{n,n}}}F(u,z)du}R_{s,t_1,\ldots,t_{n}}(z),
\end{align*}
which is exactly our assertion.
\end{proof}

 In the next theorem we give a computation of the conditional structure \eqref{cs-Psi}.
We start with an auxiliary result.
\begin{lemma}\label{rela}
Let  $X$ be a Markov process having the FOED property.
	For any $s,t > 0$ and $x\in E$ we have
\begin{align}
	p_{s+t}(x,w)\mathbb{P}(X_{s}\in dw) = p_{s}(x,w)\mathbb{P}(X_{t+s}\in dw).
\end{align}
\end{lemma}
\begin{proof}  By the FOED property, and the time homogeneity of $X$, we have for any $s,t > 0$
\begin{align*}
	\mathbb{P}(X_{s+t}\in dw) = \mathbb{P}(X_{s}\in dw)e^{\int_s^{t+s}F(u,w)du},
\end{align*}
which actually means that
\begin{align*}
	p_{s+t}(x,w)\gamma(dx) = p_{s}(x,w)e^{\int_s^{t+s}F(u,w)du}\gamma(dx).
\end{align*}
This yields $\gamma$- a.s.
\begin{align*}
	p_{s+t}(x,w) = p_{s}(x,w)e^{\int_s^{t+s}F(u,w)du},
\end{align*}
so
\begin{align*}
	p_{s+t}(x,w)\mathbb{P}(X_{s}\in dw) = p_{s}(x,w)e^{\int_s^{t+s}F(u,w)du}\mathbb{P}(X_{s}\in dw),
\end{align*}
which by the FOED property, and the time homogeneity of $X$, becomes
\begin{align*}
	p_{s+t}(x,w)\mathbb{P}(X_{s}\in dw) = p_{s}(x,w)\mathbb{P}(X_{t+s}\in dw).
\end{align*}
The proof is complete.
\end{proof}

For a transition density $p_t(\cdot,\cdot)$ and a measure $\gamma$ we will use the notation
\begin{align*}
	\gamma^{n+1}(d(z_0,\ldots,z_n)) = \gamma(dz_0)\otimes\ldots\otimes\gamma(dz_n), \qquad g_t(z) = \int_E p_t(v,z)\gamma(dv).
\end{align*}

\begin{theorem} Assume that $X$ is a Markov process having the FOED property with the initial distribution $\gamma$ and the transition density $p_t(\cdot,\cdot)$. Let $n\in\mathbb{N}$, $s\geq 0$ and
 $0 = t_0 < t_1<\ldots<t_n$. Let  $\Psi:E^{n+1}\rightarrow \mathbb{R}$ be a Borel function such that $\mathbb{E}|\Psi(X_{0},\ldots,X_{t_n})| < \infty$.  Then
\begin{align}\label{Psiforsym}
\mathbb{E}&\big(\Psi\big(X_{0},\ldots,X_{t_n}\big)\big| X_{T^s_{n,n}} = w \big)\notag\\
&= e^{-\int_{T^s_{n-1,n}}^{T^s_{n,n}}F(u,w)du}\int_{E^{n+1}}\Psi(z_0,\ldots,z_n)\prod_{i=1}^{n-1}\frac{p_{t_i}(z_{i-1},z_i)}{g_{t_i}(z_{i-1})}\frac{p_{t_n}(z_n,w)}{g_{t_n}(w)}\notag\\
&\times e^{\sum_{i=1}^n\int_0^{t_i-t_{i-1}}F(u,z_i)du}\gamma^{n+1}(d(z_0,\ldots,z_n)).
\end{align}
\end{theorem}
\begin{proof} First we prove that formula (\ref{Psiforsym}) holds for $$\Psi(z_0,\ldots,z_n) = f_0(z_0)\cdots f_n(z_n),$$ where $f_0,\ldots, f_n$ are bounded Borel functions.
	
  For a  bounded Borel function $f$ we will determine the form of $\Lambda_tf$.
By (\ref{Lmt}) and Proposition \ref{Ex1} we have, for any $t>0$ and $w\in E$,
\begin{align*}
	\Lambda_tf(w) = \int_Ef(v)\frac{p_t(v,w)\gamma(dv)}{\int_Ep_t(r,w)\gamma(dr)}.
\end{align*}
 Let us define a function $(t,w) \rightarrow \eta_{t,w}$, where $\eta_{t,w}$ is a measure on $[0,\infty)$ given by
\begin{align}\label{peta}
	\eta_{t,w}(dv) = \frac{p_t(v,w)\gamma(dv)}{\int_Ep_t(r,w)\gamma(dr)},
\end{align}
 so
\begin{align}\label{Leta}
\Lambda_tf(w) = \int_Ef(v)\eta_{t,w}(dv).
\end{align}

Now, we compute the succesive inner components of the RHS of (\ref{csfd}). By (\ref{Leta}) we have
\begin{align}\label{1igrd}
	\Lambda_{t_2-t_1}\big(f_1(\cdot)e^{\int_0^{t_1}F(u,\cdot)du}\Lambda_{t_1}f_0\big)(w) = \int_Ef_1(v)e^{\int_0^{t_1}F(u,v)}\Lambda_{t_1}f_0(v)\eta_{t_2-t_1,w}(dv).
\end{align}
Next,     for $n=3$, using   formulas (\ref{1igrd}) and (\ref{Leta}) we obtain
\begin{align*}
&\Lambda_{t_3-t_2}\big(f_2(\cdot)e^{\int_0^{t_2-t_1}F(u,\cdot)du}\Lambda_{t_2-t_1}\big(f_1(\cdot)e^{\int_0^{t_1}F(u,\cdot)du}\Lambda_{t_1}f_0\big)\big)(w) \\
&= \Lambda_{t_3-t_2}\Big(f_2(\cdot)e^{\int_0^{t_2-t_1}F(u,\cdot)du}\Big(\int_Ef_1(z_1)e^{\int_0^{t_1}F(u,z_1)du}\Lambda_{t_1}f_0(z_1)\eta_{t_2-t_1,\cdot}(dz_1)\Big)\Big)(w)\\
&= \int_E\Big(f_2(z_2)e^{\int_0^{t_2-t_1}F(u,z_2)du}\Big(\int_Ef_1(z_1)e^{\int_0^{t_1}F(u,z_1)du}\Lambda_{t_1}f_0(z_1)\eta_{t_2-t_1,z_2}(dz_1)\Big)\Big)\eta_{t_3-t_2,w}(dz_2)\\
&= \int_E\Big(f_2(z_2)e^{\int_0^{t_2-t_1}F(u,z_2)du}\Big(\int_Ef_1(z_1)e^{\int_0^{t_1}F(u,z_1)du}\\
& \qquad \qquad \qquad \qquad \times \int_Ef_0(z_0)\eta_{t_1,z_1}(dz_0)\eta_{t_2-t_1,z_2}(dz_1)\Big)\Big)\eta_{t_3-t_2,w}(dz_2),
\end{align*}
which, using Fubini theorem, we rewrite as
\begin{align*}
	\int_E\int_E\int_Ef_2(z_2)f_1(z_1)f_0(z_0)e^{\int_0^{t_2-t_1}F(u,z_2)du + \int_0^{t_1}F(u,z_1)du}\eta_{t_1,z_1}(dz_0)\eta_{t_2-t_1,z_2}(dz_1)\eta_{t_3-t_2,w}(dz_2).
\end{align*}
By formula (\ref{peta}) and the form of $g_t$ this is equal to
\begin{align*}	
\int_E&\int_E\int_Ef_2(z_2)f_1(z_1)f_0(z_0)e^{\int_0^{t_2-t_1}F(u,z_2)du + \int_0^{t_1}F(u,z_1)du}\\
&\times \quad \frac{p_{t_3-t_2}(z_2,w)p_{t_2-t_1}(z_1,z_2)p_{t_1}(z_0,z_1)}{g_{t_3-t_2}(z_2)g_{t_2-t_1}(z_1)g_{t_1}(z_0)}\gamma(dz_2)\gamma(dz_1)\gamma(dz_0).
\end{align*}
  Continuing in an analogous way we get the formula
\begin{align}\label{thetan}
	&\Lambda_{t_n-t_{n-1}}\big(f_{n-1}(\cdot)e^{\int_0^{t_{n-1}-t_{n-2}}F(u,\cdot)du}\Lambda_{t_{n-1}-t_{n-2}}\big(\ldots\big)\big)(w)\notag\\
	&= \int_E\ldots\int_Ef_{n-1}(z_{n-1})\ldots f_0(z_0)e^{\sum_{i=1}^{n-1}\int_0^{t_i - t_{i-1}}F(u,z_i)du}\frac{\prod_{i=1}^{n-1}p_{t_i - t_{i-1}}(z_{i-1},z_i)}{\prod_{i=1}^{n-1}g_{t_i-t_{i-1}}(z_{i-1})}\notag\\
	&\quad\quad \times\frac{p_{t_n - t_{n-1}}(z_{n-1},w)}{g_{t_{n}-t_{n-1}}(z_{n-1})}\gamma(dz_0)\ldots\gamma(dz_{n-1}).
\end{align}
Denote the last expression by $\Theta(n)(w)$. Using (\ref{Leta}) we compute
\begin{align*}
	\Lambda_{T^s_{n-1,n}}&\Big(f_n(\cdot)e^{\int_0^{t_n-t_{n-1}}F(u,\cdot)du}\Theta(n)(\cdot)\Big)(w)\\
	&= \int_Ef_n(z_n)e^{\int_0^{t_n-t_{n-1}}F(u,z_n)du}\Theta(n)(z_n)\frac{p_{T^s_{n-1,n}}(z_n,w)}{g_{T^s_{n-1,n}}(w)}\gamma(dz_n).
\end{align*}

Hence, using formula (\ref{thetan}) we conclude from  (\ref{csfd}) that
\begin{align*}
	\mathbb{E}&\big(f_1(X_{t_1})\cdots f_n(X_{t_n})\big| X_{T^s_{n,n}} = w \big)\\
	&= e^{-\int_{T^s_{n-1,n}}^{T^s_{n,n}}F(u,w)du}\int_{E^{n+1}}f_0(z_0)\cdots f_n(z_n)\frac{\prod_{i=1}^{n-1}p_{t_i - t_{i-1}}(z_{i-1},z_i)}{\prod_{i=1}^{n-1}g_{t_i-t_{i-1}}(z_{i-1})}\frac{p_{T^s_{n-1,n}}(z_n,w)}{g_{T^s_{n-1,n}}(w)}\notag\\
&\times e^{\sum_{i=1}^n\int_0^{t_i-t_{i-1}}F(u,z_i)du}\gamma^{n+1}(d(z_0,\ldots,z_n)).
\end{align*}
Using   Lemma \ref{rela} we find that
\begin{align*}
\frac{p_{t_i - t_{i-1}}(z_{i-1},z_i)}{g_{t_i-t_{i-1}}(z_{i-1})} = \frac{p_{t_i}(z_{i-1},z_i)}{g_{t_i}(z_{i-1})}, \qquad \quad	\frac{p_{T^s_{n-1,n}}(z_n,w)}{g_{T^s_{n-1,n}}(w)} = \frac{p_{t_n}(z_n,w)}{g_{t_n}(w)},
\end{align*}
so
  formula (\ref{Psiforsym}) holds for $$\Psi(z_0,\ldots,z_n) = f_0(z_0)\cdots f_n(z_n),$$ where $f_0,\ldots, f_n$ are bounded Borel functions.

  The general  case, that is (\ref{Psiforsym}) for a Borel functions $\Psi:E^{n+1}\rightarrow \mathbb{R}$ such that $\mathbb{E}|\Psi(X_{0},\ldots,X_{t_n})| < \infty$ follows by using a monotone convergence theorem.

 The proof is complete.
\end{proof}

\section{FOED and Kolmogorov distance: example}

In our considerations in this section we use the FOED property to solve the following problem: Consider   a Brownian motion $B$ with the  initial distribution $\gamma$  on $\mathbb{R}$. Suppose that $X$ is a Markov process having the  FOED property. Our task is to measure the similarity between   $B_t$ and $X_t$, for fixed $t$, if $X_0$ has a distribution  $\gamma$.
 In statistics, the Kolmogorov distance is a convenient way of measuring similarity of distributions (see \cite{Wil} for an overview). For a couple of nice recent results measuring the Kolmogorov distance between the general Poisson random variables and Gaussian distribution see for example \cite{4mom}.
  Recall that the Kolmogorov distance between  real valued random variables $X$ and $Y$ is defined as
\begin{align*}
	d_{Kol}(X,Y) = \sup_{x\in\mathbb{R}}\big|\mathbb{P}(X\leq x) - \mathbb{P}(Y\leq x)\big|.
\end{align*}
    We will show that for computation Kolmogorov distance $d_{Kol}(X_t,B_t)$ it suffices to consider the supremum over a special set that is related to the FOED function and the distribution  $\gamma$.  We use below the common notation: by $\Phi$ we denote the cumulative distribution function of a standard Gaussian random variable $G$ and by $\phi$ the p.d.f of $G$.

\begin{theorem}\label{KDi} Suppose that $X$ is a Markov process having the  FOED property with the function of evolution of distribution $F$, $X_0$ has a distribution $\gamma$ and a differentiable density $g$. Let $B$ be    a Brownian motion with the  initial distribution $\gamma$. Fix  $t>0$. Let $\mathcal{X}_0$ be the set of solutions of   the following equation with respect to variable~$x$
\begin{align} \label{xocon}
g(x)e^{\int_0^tF(u,x)du} = \mathbb{E}1_{\{G\leq 0\}}\big(g(x - \sqrt{t}G) + g(x + \sqrt{t}G)\big).
\end{align}
Then
 \begin{align}\label{expkol}
d_{Kol}&(X_t,B_t)\\
 &= \sup_{x_0\in \mathcal{X}_0}\Big|\int_{-\infty}^{x_0}\Big(e^{\int_0^tF(u,z)du} - \Phi\Big(\frac{x_0-z}{\sqrt{t}}\Big)\Big)g(z)dz - \int_{x_0}^{\infty}\Phi\Big(\frac{x_0-z}{\sqrt{t}}\Big)g(z)dz\Big|.\notag
\end{align}
%
\end{theorem}
\begin{proof}
For any $x\in\mathbb{R}$ we have
	\begin{align*}
		\mathbb{P}(B_t\leq x)
		= \int^{\infty}_{-\infty}\Big(\int_{-\infty}^x\frac{1}{\sqrt{2\pi t}}e^{-\frac{(z-y)^2}{2t}}dy\Big)g(z)dz
		=	\int^{\infty}_{-\infty}\Phi\Big(\frac{x-z}{\sqrt{t}}\Big)g(z)dz.
		 	\end{align*}
 Hence and from the fact that $X$ has  the FOED property we have
 \begin{align}\label{repsi}
	\psi(x) &:= \mathbb{P}(X_t\leq x) - \mathbb{P}(B_t\leq x)   =  \int_{-\infty}^xe^{\int_0^tF(u,z)du}g(z)dz - \int^{\infty}_{-\infty}\Phi\Big(\frac{x-z}{\sqrt{t}}\Big)g(z)dz,\notag
 \\& =  \int_{-\infty}^{\infty}\Big(1_{(-\infty,\ x]}(z)e^{\int_0^tF(u,z)du} - \Phi\Big(\frac{x-z}{\sqrt{t}}\Big)\Big)g(z)dz\notag\\
&= \int_{-\infty}^{x}\Big(e^{\int_0^tF(u,z)du} - \Phi\Big(\frac{x-z}{\sqrt{t}}\Big)\Big)g(z)dz - \int_x^{\infty}\Phi\Big(\frac{x-z}{\sqrt{t}}\Big)g(z)dz.
\end{align}
Substituting in the last integral $w = 2x - z$ yields
\begin{align*}
	\int_x^{\infty}\Phi\Big(\frac{x-z}{\sqrt{t}}\Big)g(z)dz = \int_{-\infty}^x\Phi\Big(\frac{w-x}{\sqrt{t}}\Big)g(2x - w)dw.
\end{align*}
Hence
\begin{align*}
	\psi(x) &= \int_{-\infty}^{x}\Big[\Big(e^{\int_0^tF(u,z)du} - \Phi\Big(\frac{x-z}{\sqrt{t}}\Big)\Big)g(z) - \Phi\Big(\frac{z-x}{\sqrt{t}}\Big)g(2x - z)\Big]dz\\
	&= \int_{-\infty}^{x}\Big[\Big(e^{\int_0^tF(u,z)du} - \Phi\Big(\frac{x-z}{\sqrt{t}}\Big)\Big)g(z) - \Big(1- \Phi\Big(\frac{x-z}{\sqrt{t}}\Big)\Big)g(2x - z)\Big]dz\\
	&= \int_{-\infty}^{x}\Big[e^{\int_0^tF(u,z)du}g(z) - \Phi\Big(\frac{x-z}{\sqrt{t}}\Big)\Big(g(z) - g(2x - z)\Big) - g(2x - z)\Big]dz.
\end{align*}
The last representation of $\psi$ clearly shows that  $\psi$ is differentiable.
 We have
\begin{align*}
	\psi'(x) = e^{\int_0^tF(u,x)du}g(x) - g(x) + \int_{-\infty}^{x}\Big[ &- \frac{1}{\sqrt{t}}\phi\Big(\frac{x-z}{\sqrt{t}}\Big)\Big(g(z) - g(2x - z)\Big)\\
	&+ 2\Phi\Big(\frac{x-z}{\sqrt{t}}\Big)g'(2x - z)- 2g'(2x - z)\Big]dz.
\end{align*}
So if $\psi'(x_0) = 0$, then
\begin{align}\label{pomxo}
	g(x_0)\Big(1 - e^{\int_0^tF(u,x_0)du}\Big) = \int_{-\infty}^{x_0}\Big[& \frac{1}{\sqrt{t}}\phi\Big(\frac{x_0-z}{\sqrt{t}}\Big)\Big(g(2x_0 - z) - g(z)\Big)\\
	&+ 2g'(2x_0 - z)\Big(\Phi\Big(\frac{x_0-z}{\sqrt{t}}\Big)- 1\Big)\Big]dz =: I + II.\notag
\end{align}
We compute $I$ substituting $w = z - x_0$
\begin{align*}
I	&= \int^0_{-\infty}\frac{1}{\sqrt{2\pi t}}e^{-\frac{w^2}{2t}}(g(x_0 - w) - g(x_0 + w))dw\\
	&= \mathbb{E}1_{\{G\leq 0\}}\big(g(x_0 - \sqrt{t}G) - g(x_0 + \sqrt{t}G)\big).
\end{align*}
 To compute $II$ we integrate by parts and substitute $w = z - x_0$
\begin{align*}
	II
	&= \Big[-2g(2x_0 - z)\Big(\Phi\Big(\frac{x_0-z}{\sqrt{t}}\Big)- 1\Big)\Big]^{x_0}_{-\infty} - 2\int_{-\infty}^{x_0}g(2x_0 - z)\frac{1}{\sqrt{t}}\phi\Big(\frac{x_0-z}{\sqrt{t}}\Big)dz\\
	&= g(x_0) - 2\int_{-\infty}^0\frac{1}{\sqrt{2\pi t}}e^{-\frac{w^2}{2t}}g(x_0 - w)dw = g(x_0) - 2\mathbb{E}1_{\{G\leq 0\}}g(x_0 - \sqrt{t}G).
\end{align*}
 Putting together just computed $I$, $II$ and comparing their sum with the LHS of (\ref{pomxo}) yields (\ref{xocon}).
From this and  from the fact that  $\lim_{x\rightarrow -\infty}\psi(x) = \lim_{x\rightarrow \infty}\psi(x) = 0$, we obtain
\begin{align*}
	\sup_{x\in\mathbb{R}}|\psi(x)| = \Big|\int_{-\infty}^{x_0}\Big(e^{\int_0^tF(u,z)du} - \Phi\Big(\frac{x_0-z}{\sqrt{t}}\Big)\Big)g(z)dz - \int_{x_0}^{\infty}\Phi\Big(\frac{x_0-z}{\sqrt{t}}\Big)g(z)dz\Big|,
\end{align*}
for some $x_0$ solving equation (\ref{xocon}). This proves (\ref{expkol}), which finishes the proof.
\end{proof}

\bibliographystyle{plain}

\begin{thebibliography}{10}
\bibitem{BRY} Barrieu P., Rouault A., Yor M.
\textit{A study of the Hartman-Watson distribution motivated by
numerical problems related to the pricing of Asian options}, J.
Appl. Probab.  41,  1049-1058 (2004).

\bibitem{Ber} Berestycki, H., J. Busca, and I. Florent \emph{Computing the implied volatility in stochastic
volatility models}, Communications on Pure and Applied Mathematics 57(10), 1352–1373 (2004).

\bibitem{SB} Borodin A., Salminen P., \textit{Handbook of Brownian Motion - Facts and Formulae}, Birkh\"{a}user (2nd
ed.), 2002.

 \bibitem{Car} Carmona R., Durrleman V.  \textit{Pricing
and hedging spread options}, SIAM Rev.  45,   627-685 (2003).

\bibitem{4mom} Dobler Ch., Peccati G. \textit{The fourth moment theorem on the Poisson space},
Ann. Probab. 46 (4)  1878-1916 (2018).

\bibitem{E} Eqworld - The World of Mathematical equations, \\ http://eqworld.ipmnet.ru/index.htm

\bibitem{Fo} Fouque Jean-Pierre, George Papanicolaou, and K Ronnie Sircar, \emph{Derivatives in financial markets with stochastic volatility}, Cambridge University Press, (2000).

\bibitem{Fo2} Fouque, J.-P., G. Papanicolaou, R. Sircar, and K. Solna  \emph{Multiscale stochastic volatility for
equity, interest rate, and credit derivatives}, Cambridge: Cambridge University Press, (2011).

 \bibitem{G} Gulisashvili, A. \textit{Analytically Tractable Stochastic Stock Price Models}, Springer Finance, 2012.

\bibitem{GH} Gerhold, S. \textit{The Hartman-Watson distribution revisited: asymptotics for pricing Asian options}, J. Appl. Probab. Volume 48, Number 3,  892-899 (2011).


\bibitem{Hag1} Hagan P., Kumar D., Lesniewski A., Woodward D.  \textit{Managing
smile risk.} Wilmott Magazine, September,  84-108 (2002).

\bibitem{Hul} Hull, J., White, A.  \textit{The pricing of options
on assets with stochastic volatilities.} J. Finance 42, 281-300
(1987).


\bibitem{Ike} Ikeda N., Watanabe S.  \textit{Stochastic Differential Equations and
 Diffusion Processes.} North-Holland Kodansha 1981.


\bibitem{JWloc1} Jakubowski J., Wisniewolski M., \emph{Another Look at the Hartman-Watson Distributions},
Potential Anal 53, 1269–1297 (2020).

\bibitem{JWlt} Jakubowski J., Wisniewolski M., \emph{A convolution formula for the local time of an It\^o diffusion reflecting at $0$ and  a generalized Stroock-Williams equation},  Bernoulli 27(3), 2021, 1870–1898.

\bibitem{JW09} Jakubowski J.,  Wi\'sniewolski M. \textit{Probabilistic representations of the
density function of the asset price and of vanilla
options in linear stochastic volatility models.}  (2009) arXiv:0909.4765.

\bibitem{JW20} Jakubowski J.,  Wi\'sniewolski M. \textit{
Revisiting  linear and lognormal stochastic volatility models}. Banach Center Publications 122, 169-185 (2020).

\bibitem{Jou} Jourdain B.  \textit{Loss of
martingality in asset price model with log-normal stochastic
volatility.} ENPC-CERMICS, Working paper (2004).

\bibitem{Lac}  Lacour C. \emph{Nonparametric estimation of the stationary density and the transition density of a Markov chain},
Stochastic Processes and their Applications,
Volume 118, Issue 2,
2008,
Pages 232-260,

\bibitem{Lbdv} Lebedev, N.N. \textit{Special Functions and their Applications}, Dover,(1972).
New York.

\bibitem{Mik} Mikosch, T. \textit{Elementary Stochastic Calculus with Finance in View}, World-Scientific (1998).

\bibitem{Mu} Musiela M., Rutkowski M.  \emph{Martingale methods in financial
modelling}, Vol. 36. Springer Science and Business Media, (2006).

\bibitem{Mat} Matsumoto H., Yor M.  \textit{Exponential
functionals of Brownian motion, I, Probability laws at fixed time.}
Probab. Surveys 2, 312-347 (2005).

\bibitem{RE} Rebonato R.  \textit{
Volatility and Correlation. The Perfect Hedger and the Fox.} Wiley
(2nd ed.) 2004.

\bibitem{RY} Revuz D., Yor M.  \textit{Continous Martingales and Brownian Motion}. Springer-Verlag (3rd
ed.). 2005.

\bibitem{Ru} Rudin W., \emph{Functional Analysis}, McGraw-Hill (2nd ed.), 1991.

\bibitem{Sar} Sart M. \emph{Estimation of the transition density of a Markov chain}, Ann. Inst. H. Poincaré Probab. Statist. 50 (3) 1028 - 1068, August 2014.

\bibitem{Sin} Sin C. \textit{Complications with stochastic volatility models}.
Adv. in Appl. Probab.  30,  256-268, (1998).

\bibitem{Sche} Sheu, S.-J. \emph{Some Estimates of the Transition Density of a Nondegenerate Diffusion Markov Process}, The Annals of Probability, 19(2), 538–561, (1991)

\bibitem{SW} Stanislaw, J. S., Werner, E. (1999). \emph{A nonsymmetric correlation inequality for Gaussian measure},
Y. Multivariate Anal., 68, 193-211.

\bibitem{Wil} Wilcox R. \emph{Introduction to Robust Estimation and Hypothesis Testing}, Academic Press (Fourth Edition), 2017.


\end{thebibliography}

\end{document}